\documentclass[a4paper,11pt]{article}

\usepackage{amsmath}
\usepackage{amsfonts}
\usepackage{amssymb}
\usepackage{amsthm}
\usepackage{graphicx}
\usepackage{verbatim}
\usepackage{color}
\usepackage[numbers]{natbib}
\usepackage{bibentry}

\usepackage{authblk}
\usepackage{setspace}

\newtheorem{thm}{Theorem}[section]
\newtheorem{cor}[thm]{Corollary}
\newtheorem{lem}[thm]{Lemma}
\newtheorem{prop}[thm]{Proposition}
\newtheorem{conj}{Conjecture}[section]

\theoremstyle{definition}
\newtheorem{defin}[thm]{Definition}

\newtheorem{rem}[thm]{Remark}

\newcommand{\eps}{\varepsilon}
\newcommand{\deltabf}{\text{\boldmath $\delta$}}
\newcommand{\B}{\mathbb{B}}
\newcommand{\C}{\mathbb{C}}

\newcommand{\N}{\mathbb{N}}

\newcommand{\R}{\mathbb{R}}
\newcommand{\T}{\mathbb{T}}
\newcommand{\W}{\mathbb{W}}
\newcommand{\Z}{\mathbb{Z}}
\newcommand{\AAA}{\mathcal{A}}
\newcommand{\BB}{\mathcal{B}}
\newcommand{\DD}{\mathcal{D}}

\newcommand{\HH}{\mathcal{H}}
\newcommand{\LL}{\mathcal{L}}
\newcommand{\MM}{\mathcal{M}}

\newcommand{\RR}{\mathcal{R}}
\newcommand{\RHH}{\mathcal{RH}}
\newcommand{\TT}{\mathcal{T}}
\newcommand{\XX}{\mathcal{X}}

\newcommand{\Cl}{\text{Cl}}
\newcommand{\Spec}{\text{Spec}}

\newcommand{\ignore}[1]{}
\newcommand{\nobibentry}[1]{{\let\nocite\ignore\bibentry{#1}}}

\author[1]{Pau Rabassa}
\author[2]{Angel Jorba}
\author[2]{Joan Carles Tatjer}
\affil[1] {\mbox{Johann Bernoulli Institute for Mathematics and Computer Science,}
\centerline{ \mbox{University of Groningen, Groningen, The Netherlands}}
\newline
\mbox{E-mail: {\tt paurabassa@gmail.com}}
\vspace{2mm}}
\affil[2]{\mbox{Departament of Matem\`atica Aplicada i An\`alisi,}
\mbox{Universitat de Barcelona, Barcelona, Spain}
\mbox{E-mails: {\tt angel@maia.ub.edu}, {\tt jcarles@maia.ub.es}}}

\date{}

\title{
Towards a renormalization theory for quasi-periodically forced 
          one dimensional maps {I}. {E}xistence of reducibility loss 
          bifurcations\thanks{This work has been supported by the 
          MEC grant MTM2009-09723 and the CIRIT grant 2009 SGR 67.
          P.R. has been partially supported by the PREDEX project, funded
          by the Complexity-NET: {\tt www.complexitynet.eu}.}
}
\begin{document}
\maketitle
\begin{abstract}
We propose an extension of the one dimensional (doubling) 
renormalization operator to the case of maps on the cylinder. 
The kind of maps considered are commonly referred as quasi-periodic 
forced one dimensional maps. 
We prove that the 
fixed point of the one dimensional renormalization operator 
extends to a fixed point of the quasi-periodic forced 
renormalization operator. We also prove that 
the operator is differentiable around the fixed point and we 
study its derivative. Then we consider a two parametric family 
of quasi-periodically forced maps which is a unimodal one dimensional 
map with a full cascade of period doubling bifurcations 
plus a quasi-periodic perturbation. For one 
dimensional maps it is well known that between one period 
doubling and the next one there exists a parameter value 
where the $2^n$-periodic orbit is superatracting. 
Under appropriate hypotheses, we prove that the 
two parameter family has two curves of reducibility loss bifurcation
around these points. 
\end{abstract}

\tableofcontents


\section{Introduction}
\label{chapter renormalization for q.p. forced maps}

This is the first of a series of papers (together with \cite{JRT11b,JRT11c})
proposing an extension of the one dimensional renormalization 
theory for the case of quasi-periodic forced maps. Each of these 
papers is self contained, but highly interrelated with the others. A more 
detailed exposition can be found in \cite{Rab10}.  
In this paper we give a concrete definition of the operator to the case of 
quasi-periodic maps and we use it to prove the existence of 
reducibility loss bifurcations when the coupling parameter goes to 
zero, like the ones observed in \cite{JRT11p, JRT11p2} 
for the Forced Logistic map. 
In \cite{JRT11b} we will use the results developed here to study the 
asymptotic behavior of these bifurcations when the period of the attracting set 
goes to infinity. Our quasi-periodic extension of the renormalization 
operator is not complete in the sense that 
several conjectures must be assumed. In \cite{JRT11c} we include 
the numerical evidence which support these conjectures and we show 
that the theoretical results agree with the behavior observed numerically.

The classic one dimensional renormalization theory was motivated to explain
the cascades of period doubling bifurcations. 
The paradigmatic example in the case of unimodal maps is the
Logistic Map, 
but the properties of renormalization and universality
are also observable in a wider class of maps.
Concretely, given a typical one parametric family for unimodal maps
$\{f_\alpha\}_{\alpha\in I}$  one  observes
numerically that there exists a sequence of parameter
values $\{d_n\}_{n\in \N} \subset I$ such that
the attracting periodic orbit of the map undergoes a period doubling 
bifurcation. Between one period doubling and the next one there 
exists also a parameter value $s_n$, for which the critical point 
of $f_{s_n}$ is a periodic orbit with period $2^n$. One can 
also observe that 
\begin{equation}
\label{universal limit sumicon}
\lim_{n\rightarrow \infty} \frac{d_n - d_{n-1}} {d_{n+1} - d_{n}} = 
\lim_{n\rightarrow \infty} \frac{s_n - s_{n-1}} {s_{n+1} - s_{n}} =
\deltabf  = \texttt{ 4.66920...}. 
\end{equation}

Moreover, the constant $\deltabf$ is universal,
in the sense that for any family  of unimodal maps with a quadratic
turning point having a cascade of period doubling bifurcations,
one obtains the same ratio $\deltabf$. For technical 
reason the discussion is typically focussed 
around the values $s_n$.

The renormalization theory for unimodal one dimensional maps 
was originated by the seminal works of Feigenbaum 
(\cite{Fei78,Fei79}) and Collet and Tresser (\cite{CT78}) who independently 
proposed the renormalization operator to explain the universal
behavior observed in the cascades of bifurcations of
one dimensional maps, see \cite{MvS93} for a review. 
Let us do a quick summary of the theory. The (doubling) renormalization
operator, which is denoted by $\RR$, is defined in the space of unimodal maps 
as the self composition of the map composed with a change of 
scale (see subsection \ref{subsection Set up 1-d renor} for more 
details). There are some basic assumptions on the dynamics of the operator 
$\RR$ (known as the Feigenbaum conjectures) which give a suitable 
explanation to the universality described before. 
The first of these conjectures is that the operator has a fixed 
point $\Phi$ and it is differentiable in a neighborhood of $\Phi.$ 
The second conjecture is that the spectrum of $D\RR(\Phi)$ has a  
a unique real eigenvalue $\deltabf$ bigger than one, and the rest of
eigenvalues are strictly smaller than one. Then one has that 
the unstable manifold $W^u(\Phi,\RR)$ has dimension one and the stable manifold 
$W^s(\Phi,\RR)$ has codimension one.

On the other hand, one has that the values $\alpha=s_n$
where the critical point
of a map $f_{s_n}$ has period $2^n$ correspond to the
parameter values where the family $\{f_\alpha\}_{\alpha\in I\subset\R}$
intersects certain codimension one manifolds $\Sigma_n$. 
Moreover one has that $\RR(\Sigma_n) \subset \Sigma_{n-1}$ .
The third conjecture claims that these manifolds intersect 
transversally $W^u(\Phi,\RR)$. Then one has that they accumulate to 
$W^s(\Phi,\RR)$ with ratio $\deltabf$ in a neighborhood $U$ of $\Phi$. 
Then, for any one dimensional family $\{f_\alpha\}_{\alpha\in I}$ intersecting
$W^s(\Phi,\RR)\cap U$ transversally, the family has
a sequence $\{s_n\}_{n\in \N}$ of parameters where the critical
 point of the map $f_{s_n}$ is periodic with
period $2^n$. Hence, applying the $\lambda$-lemma, this parameter values  
satisfy the asymptotic behavior
given by the equation (\ref{universal limit sumicon}), with $\deltabf$ the 
unstable eigenvalue of $D\RR(\Phi)$.

The first proofs of the Feigenbaum conjectures were done  with 
computer assistance (\cite{Lan82,EW87}). Later on completely conceptual proofs 
appeared (\cite{Sul92,Lyu99}), all of them for the case of analytic maps. 
For studies of the operator in the $C^r$ context see \cite{FEM06} and 
\cite{CMMT09}. Our extension to the quasi-periodic case does 
not cover all the theory exposed in the cited works.

\begin{figure}[t]
\centering
\resizebox{11cm}{!}{
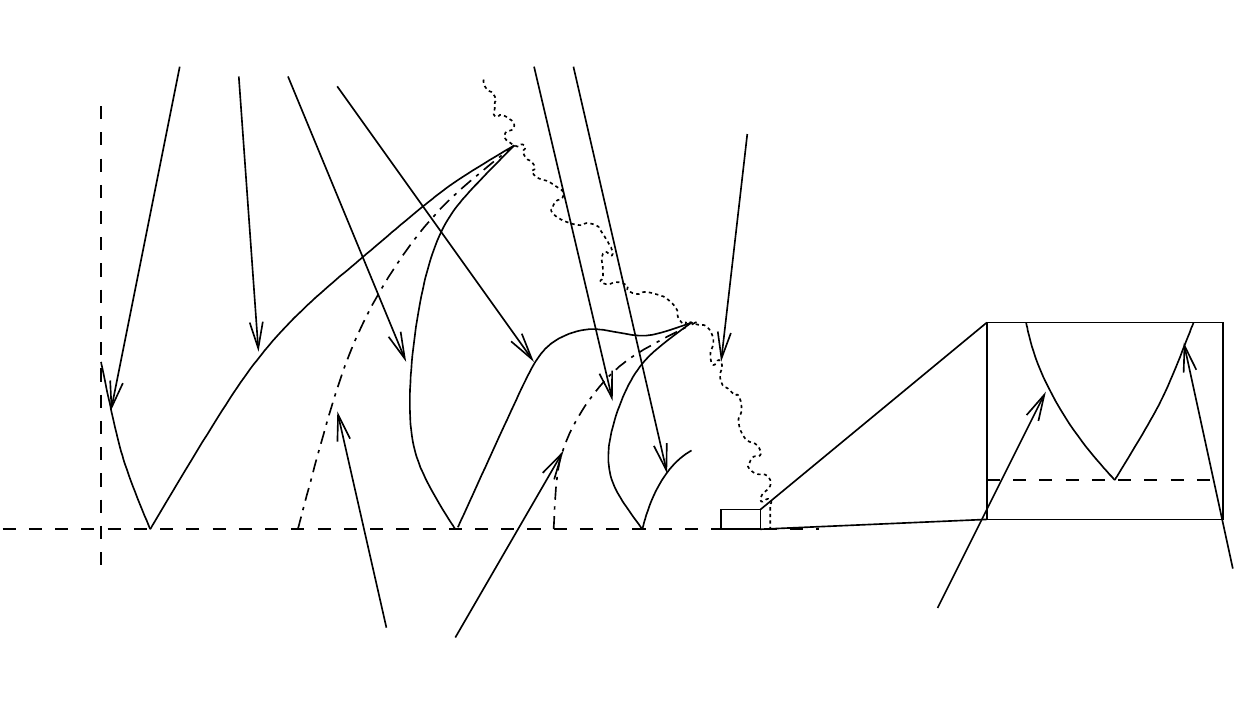
\phantom{aaaaa}
}
\caption{
Schematic representation of the bifurcations diagram 
of the Forced Logistic Map, see \cite{JRT11p} for the numeric 
computation of this diagram.} 
\label{Esquema BifPar}
\end{figure}

The paradigmatic example in this paper is the Forced Logistic Map (FLM for 
short). Nevertheless the obtained results are applicable to a wider class 
of maps. The FLM is a map in the cylinder $\T\times \R$
defined as
\begin{equation}
\label{FLM}
\left.
\begin{array}{rclcl}
\bar{\theta} & = & r_\omega(\theta) & = & \theta + \omega  ,\\
\bar{x} & = & f_{\alpha,\eps} (\theta,x) & = & \alpha x(1-x)(1+ \eps \cos(2\pi \theta)) ,
\end{array}
\right\}
\end{equation}
where $(\alpha,\eps)$ are parameters and $\omega$ a fixed Diophantine
number. 
The dynamics on the periodic component is a rigid rotation
and the dynamics on the real component is the Logistic Map
plus a perturbation depending on the periodic one. Sometimes the 
FLM is also defined with $f_{\alpha,\eps} (\theta,x) = 
\alpha x(1-x) +  \eps \cos(2\pi \theta)$. The results in this 
paper applies to both cases. 

The FLM map appears in the literature in different
contexts, usually related with the destruction of invariant curves, 
see \cite{JRT11p} and references therein. 
Concretely, we are interested on the truncation of the
period doubling cascade. As discussed above, the Logistic Map 
exhibits an infinite cascade of period doubling bifurcations
which leads to chaotic behavior. For zero coupling ($\eps=0$), 
these periodic orbits become invariant curves of the FLM (provided the rotation 
number $\omega$ is irrational). But when the coupling 
parameter is different from zero, the number of period doubling
bifurcations of the invariant curves is finite.

We studied numerically this phenomenon in \cite{JRT11p}. 
Concretely,  we computed some bifurcation diagrams in terms
of the dynamics of the attracting set, taking 
into account different properties,  as the Lyapunov exponent  
and, in the case of having a periodic invariant curve, its 
period and its reducibility. In the case of analytic maps in 
the cylinder, the reducibility 
loss of an invariant curve can be characterized as a bifurcation 
(see definition 2.3 in \cite{JRT11p}). But the reducibility loss 
is not a bifurcation in the classical sense because there are not 
visible changes in the phase space, only the spectral properties 
of the transfer operator associated to the continuation of that 
curve changes (see \cite{JT08}). Despite of this, we will consider 
it as a bifurcation for the rest of this paper. The numerical 
computations in the cited work reveal that the parameter 
values for which the invariant curve doubles its period are
contained in regions of the parameter space where the invariant 
curve is reducible. As before, let $s_n$ be the parameter values where
the critical point of the uncoupled family is periodic with
period $2^n$. The numerical computations also revealed that from every 
parameter  value $(\alpha,\eps) = (s_n, 0)$ 
two curves are  born. These curves correspond to a
reducibility-loss bifurcation of the $2^n$-periodic invariant curve.
The scenario is sketched in figure \ref{Esquema BifPar}. 

Assume that these two curves can be locally expressed as 
$(\alpha_n^+(\eps),\eps)$  and $(\alpha_n^-(\eps), \eps)$
with $\alpha_n^+(0)= \alpha_n^-(0)=s_n$. 
In theorem \ref{thm existence of non-reducibility direction} we 
prove that these curves really exist for suitable 
families of maps. Moreover, we give the values of 
$\frac{d}{d \eps} \alpha ^{+}_n(0)$ and $\frac{d}{d \eps} \alpha ^{-}_n(0)$ 
in terms of the iterates of the renormalization operator. 
To achieve this result we need to assume that the quasi-periodic 
renormalization operator is injective. This assumption will be 
called Conjecture {\bf \ref{conjecture H2}} which is 
stated in section \ref{section boundaris of reducibility}. 
This conjecture will be supported numerically in \cite{JRT11c}. 
	
The paper is structured as follows. In section \ref{section definition 
and basic properties} we propose a definition for the q.p. renormalization 
operator and we study the operator as a map on the Banach space of 
q.p. forced unimodal maps. Among other results, we prove that the 
fixed point of the one dimensional renormalization operator 
extends to the quasi-periodic one and we compute and 
study its derivative. In section \ref{chapter application} we consider 
certain codimension one manifolds, which correspond to 
the bifurcation manifold associated to the reducibility loss of the 
$2^n$ periodic invariant curve of the system. We 
relate these manifolds for different values of $n$ by 
means of the renormalization operator. Then we consider 
a generic two parametric family such that it becomes 
a full family of renormalizable one dimensional maps when 
one of its parameters is equal to zero. We use the q.p. renormalization 
theory to prove the existence of reducibility loss bifurcations 
for the family. We also include an appendix where we 
analyze the minimum function as a functional operator from the space of functions 
$f:\T\rightarrow \R$  to $\R$, which is necessary for our discussion.

\section{Definition of the operator and basic properties} 
\label{section definition and basic properties}


Consider a q. p. forced map as follows,
\begin{equation}
\label{q.p. forced system interval}
\begin{array}{rccc}
F:& \T\times I &\rightarrow & \T \times I \\
  & \left( \begin{array}{c} \theta \\ x \end{array}\right)
  & \mapsto
  & \left( \begin{array}{c} \theta + \omega \\ f(\theta,x) \end{array}\right),
\end{array} 
\end{equation}
where $I=[-1,1]$ and $f\in C^r(\T\times I ,I)$. 
To define the renormalization operator we only require $r \geq 1$, 
but in this work we focus on the simplest case of analytic functions. 

Let us remark that, in this section,  no additional assumptions will be done on 
$\omega$. 
The aim of this section is to define the quasi-periodic renormalization 
operator. As long as the dynamics of the map $F$ is not considered, it 
is not necessary any additional requirement on $\omega$. In section 
\ref{chapter application} we will require $\omega$ to be Diophantine. But for 
this section it is advantageous to define 
the operator for any $\omega\in \T$, since then the operator will depend 
continuously on $\omega$. 

The definition of the renormalization operator will be done from a 
perturbative point of view. In other words, we will consider a 
map $F$ like (\ref{q.p. forced system interval})
 such that $f(\theta,x)= f_0(x) + h(\theta,x)$, with $f_0$ a
unimodal map on $\DD(\RR)$ the domain of the renormalization operator. In 
this section we will see that if $h$ small enough 
(in $\|\cdot\|_\infty$ norm), then we can define a ``renormalization'' of $f$. 

In section \ref{subsection Set up 1-d renor} we introduce the setup 
of the one dimensional renormalization operator that we 
consider in this paper.  
In section \ref{section definition of q.p. renormalization}
we define the 
renormalization operator for q.p. forced maps and we will check that the
definition is consistent. 
In section \ref{subsection study of TT} the basic properties of the operator 
are studied. Concretely, we  check that a fixed point
of the one dimension renormalization operator extends to a fixed point 
of the q.p. one and  we will also check that the operator is differentiable
in a neighborhood of the fixed point. 
In subsection \ref{section The Fourier expansion of DT}   
we study the derivative of the operator with respect 
to the Fourier expansion of the function to which the operator is applied. 
%


\subsection{Setup of the one dimensional renormalization operator. }
\label{subsection Set up 1-d renor}  

We introduce here the precise definition of the 
one dimensional renormalization operator. 
The approach chosen here is due to its
simplicity, which makes easier to adapt to the quasi 
periodic case. Concretely we follow \cite{Lan82}, but we do 
a slight modification on the domain of the operator 
for technical reasons. For a set up in a much
more general context see \cite{Sul92, MvS93, Lyu99}
and for more recent works on renormalization of one-dimensional maps
see \cite{FEM06} and \cite{CMMT09}.

For the understanding of this section it is advisable to have 
some familiarity with the definition of the (doubling) renormalization 
operator given in \cite{Lan82}. Note that the definition given there 
is for even maps defined on the interval $[-1,1]$, such that the 
turning point is $0$ and it is mapped to $1$.  Given a skew map 
$F$ like  (\ref{q.p. forced system interval}), we
want to define the renormalization of the map in a similar way
to the one-dimensional case. That is, to give some 
generic conditions on $F$ in such a way that 
it has a two periodic invariant subset, and such that $F^2$ restricted
to this subset is affinely conjugate to a map in the same class of 
functions of $F$. Note that the $\theta$-component of $F$, when 
$\omega$ is irrational,  does not allow the map to have invariant 
subsets in this component. Moreover we want 
the renormalization of $F$ to have a rigid rotation in the 
periodic component. Then the affine conjugacy 
should be of the form $A(\theta, x)= (\theta, a x)$, with $a$ a real number. 
Note that the skew map $A^{-1}\circ F\circ F \circ A$ has 
rotation number equal to $2\omega$ and is defined by 
the function $\frac{1}{a} f(\theta +\omega, f(\theta, a x))$, 



Suppose that we have $g$ 
a renormalizable one dimensional map 
and $h \in C^{r}(\T\times I, I)$ such that its supremum norm $\|h \|$ 
is small. Then we would like to consider a map $F$ like  
(\ref{q.p. forced system interval}) 
defined by  the function
\[f(\theta,x) = g(x) + h(\theta,x).\]

Note that the definition of the renormalization operator 
given in \cite{Lan82} is for even maps defined on the 
interval $[-1,1]$, such that the turning point is $0$ and it is mapped to $1$.
If we want $F:\T \times I \rightarrow \T \times I$ to be well defined 
we should have $f(\T\times I) \subset I$. 
Although we allow $\|h\|$ to be small,
we should require  $h(\theta,0)$ to be negative for any $\theta\in\T$.
This would make the construction quite artificial and not 
applicable to the general q.p. forced maps like the FLM. 
A solution  to this problem
is to replace the interval $I=[-1,1]$ by a wider one 
$I_\delta=[-1-\delta, 1+\delta]$, but then 
we have to check that the one dimensional renormalization operator
can be extended to this new domain.

\begin{figure}[t]
\centering
\resizebox{7.5cm}{!}{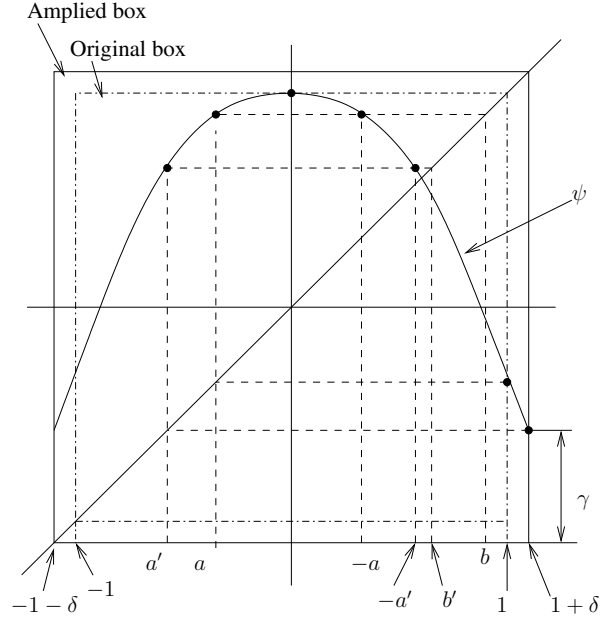}
\caption{Given a function $\psi \in \MM_\delta$, in the picture 
we show the geometrical meaning of the new constants $a'$, $b'$, $\delta$ and 
$\gamma$.}
\label{MapLanf modified}
\end{figure}

Given a small value $\delta$, let  $\MM_\delta$ denote the space 
of continuously differentiable even maps $\psi$ of the interval 
$I_\delta=[-1-\delta,1+\delta]$ into itself such that
\begin{enumerate}
\item $\psi(0)=1$,
\item $x \psi'(x) <0$ for $x\neq 0$.
\end{enumerate}

Set $a=\psi(1)$, $a'= (1+\delta)a$ and $b'=\psi(a')$. In 
figure \ref{MapLanf modified} one can see an example of a function in 
$\MM_\delta$ where these values are shown. Now, 
we define $\DD(\RR_\delta)$ as the set of $\psi \in \MM_\delta$ 
such that
\begin{enumerate}
\item  $a<0$
\item  $1> b'>-a'$,
\item  $\psi(b') <- a'$.
\end{enumerate}

\begin{rem}
\label{remmark open domain 1d} 
Note that the values $a$, $a'$ and $b'$ can be seen as continuous 
functions from $\MM_\delta$ to $\R$, and therefore the set $\DD(\RR_\delta)$ 
is open in $\MM_\delta$ (with the $C^k$ topology). 
\end{rem}

We define the renormalization operator, $\RR_\delta : 
\DD(\RR_\delta) \rightarrow \MM_\delta$ as 
\begin{equation}
\label{renormalization operator lanford}
\RR_\delta(\psi) (x) =  \frac{1}{a} \psi \circ \psi (a x).
\end{equation}
where $a=\psi(1)$.

\begin{prop}
\label{proposition RRdelta well defined} 
The operator $\RR_\delta$ is well defined, in the sense that 
$\RR_\delta(\psi)$ belongs to $\MM_\delta$ for any $\psi \in \DD(\RR_\delta)$ 
\end{prop}


\begin{prop}
\label{proposition fixed points extend to RRdelta}
Any fixed point $\Phi\in \MM_0$ of $\RR_0$ extends to a fixed 
point of $\RR_\delta$, as long as $\Phi\in \DD(\RR_\delta)$. 
\end{prop}

Let us remark that in the proof of the existence of a fixed 
point done in \cite{Lan82} it is claimed  that there exists 
a function $\Phi$, which is analytic on the 
domain $\left\{z \in \C | \thinspace |z^2-1|\leq \sqrt{8} \right\}$,  
and such that its restriction to $I=[-1,1]$ is a fixed point of $\RR$. 
Then the fixed point $\Phi$ extends to a fixed point of 
$\RR_\delta$ as long as $1+\delta < \sqrt{8}$ and 
$\Phi(1+\delta) > -(1+\delta)$. Concretely we have that there exist 
$\delta_0$ (small enough) such that 
$\Phi$ extends to a fixed point of $\RR_\delta$ for any  $0<\delta<\delta_0$.
For the rest of this paper $\delta$ is fixed equal to $\delta_0$.

\subsubsection{Proofs}

\begin{proof}[Proof of proposition \ref{proposition RRdelta well defined}]
Given a function $\psi\in \DD(\RR_\delta)$, let $\Psi=\RR_\delta(\psi)$. 
Then 
\[
\Psi(0) = \frac{1}{a} \psi \circ \psi (0)=\frac{1}{a} \psi(1) = 1.
\]

It is easy to check that 
\begin{equation}
\label{derivada Psi}
x \Psi'(x) = x \psi'(ax) \psi'(\psi(ax))
\end{equation}

Then, using that $x\psi'(x) < 0$, for any $x\in I_\delta$ and $x\neq0$, 
we have that $x\psi'(-ax) \geq 0$, for any $x\in I_\delta$ and $x\neq 0$. 
On the other hand, for any $x \in I_\delta$ we have that $\psi(-ax) \in [b',1]$.
Using again $x\psi'(x) < 0$ and $0<a<a'<b'$ we have that $\psi'(\psi(ax)) < 0$ 
for any $x\in I_\delta$. Then we can conclude that $x \Psi'(x) < 0 $ for any 
$x\in I_\delta$ and $x\neq 0$. 

The last condition to check is that for any $x\in I_\delta$, the map 
$\Psi$ maps $x$ inside the set $I_\delta$. Using  
that $\Psi(0)=1$ and the monotonicity consequences of  $x\Psi'(x)<0$ we have
that $\Psi(x)<\Psi(0)=1$. We only have to check that $\Psi(1+\delta) > 
-1 -\delta$. Since $\Psi(1+\delta)= -\frac{1}{a} \psi(b')$, then 
the inequality holds from $\psi(b') < a' $. 
\end{proof} 

\begin{proof}[Proof of proposition 
\ref{proposition fixed points extend to RRdelta}]
A fixed point of the renormalization 
operator can be extended to the real line using recursively 
the invariance equation  $\psi(x)=\frac{1}{a}\psi\circ \psi(ax)$ to 
evaluate points from outside of $I$ (recall that $|a|<1$). 
Moreover we have to check that $\psi\in \DD(\RR_\delta)$. 
Using again that $|a|<1$, we have that $a=f(1)> -1$, 
therefore for a sufficiently small $\delta$ we will have 
$f(1+\delta) > -1-\delta$.
\end{proof}

\subsection{The renormalization operator for  
quasi-periodically forced maps}
\label{section definition of q.p. renormalization} 

In this section we define the quasi-periodic renormalization 
operator and we check that there exists a non-empty set 
of maps where it is well defined.

Consider the operator
\begin{equation}
\label{equation projection 0}
\begin{array}{rccc}
p_0:& C^r(\T \times I_\delta, I_\delta) &\rightarrow & C^r(I_\delta, I_\delta) \\
\displaystyle \rule{0pt}{3ex} & f(\theta,x) & \mapsto &  \rule{0pt}{3ex}
\displaystyle  \int_{0}^{1} f(\theta, x) d\theta .
\end{array} 
\end{equation}
If we consider the natural inclusion of $C^r(I_\delta, I_\delta)$ 
into $C^r(\T \times I_\delta, I_\delta)$ then we have that $p_0$ is a projection 
($(p_0)^2 = p_0$).

Consider $\MM_{\delta}\subset C^r(I_\delta, I_\delta)$ defined in the 
previous subsection. Then we can consider  
the space $\XX$ defined as: 
\[
\XX= \{ f \in C^r(\T \times I_\delta, I_\delta) | \thinspace 
 p_0(f) \in \MM_\delta\}, 
\]
and the decomposition $\XX=\XX_0 \oplus \XX_0^c$ of it given by the projection 
$p_0$, i.e $\XX_0=\{f \in \XX \thinspace |\text{ } p_0(f)=f\}$ and 
$\XX^c_0=\{f \in \XX \thinspace | \text{ } p_0(f)=0\}$. Note 
that from the definition of $\XX$ it follows that $\XX_0$ is an isomorphic 
copy of $\MM_\delta$.


\begin{prop}
\label{condition to belong  in X}
Let $f$ be a function in $\XX_0$ and consider 
$\gamma := f(1+\delta) +1+\delta $ (see figure \ref{MapLanf modified}). 
Consider also $h(x,\theta) \in C^{r}(\T\times I_\delta, I_\delta)$ with 
$p_0(h)=0$. If $\|h\|< \delta$ and $\|h\| < \gamma$, then $g=f+h$ belongs to $\XX$. 
\end{prop}

In other words, we have that small quasi-periodic perturbations of 
an uncoupled one dimensional map of $\MM_\delta$ belong to 
the space considered here. With this space of functions we are able to 
define the q.p. renormalization operator. The proof of proposition 
\ref{condition to belong  in X} is at the end of the subsection. 

\begin{defin}
Given a function $g\in \XX$, we can 
define 
the {\bf renormalization} of $g$ as 
\begin{equation}
\label{operator tau}
[\TT_\omega(g)](\theta,x) :=  \frac{1}{\hat{a}} g(\theta + \omega,
g(\theta, \hat{a}x)),
\end{equation}
where $\displaystyle \hat{a} = \int_{0}^{1} 
g(\theta, 1) d\theta$. Consider the set $\DD(\TT)=
\{g \in \XX \thinspace | \text{ } \TT_\omega(g) \in \XX \}$,
then the renormalization operator 
$\TT_\omega$ is defined from $\DD(\TT)$ to $\XX$. 
\end{defin}

\begin{rem}
The choice of  $\hat{a} = \int_{0}^{1} 
g(\theta, 1) d\theta$ is somehow arbitrary.  
In a more general context, one might allow $\hat{a}$ to be indeed 
a function of the angle, but the first problem arises with 
the choice of this function. We  have chosen
$a$ to be constant in order to keep the problem as simple as possible.
For our purposes we find this approach sufficient, since we 
are always considering maps which are perturbations of 
one dimensional maps.
\end{rem}

Note that, with the definition given above, it might happen that 
$D(\TT)=\emptyset$.  Then let us ensure that this is not the case. 

The space $\XX_0$, which is an isomorphic copy $\MM_\delta$, 
is a subspace of $\XX$. Then let us define $\DD_0(\TT)$ as 
the set of functions in $\XX_0$ that correspond 
to functions of $\DD(\RR_\delta) \subset \MM_\delta$. 
For any function $g\in \DD_0(\TT)$ we have 
$\TT_\omega(g) = \RR_\delta(g) \in \XX$, therefore 
$\DD_0(\TT)$ is a subset of $\DD(\TT)$ and consequently 
$\DD(\TT)$ is not empty.  Indeed we have 
the following result on the topology of $\DD(\TT)$.

\begin{prop}
\label{proposition existence of open set domain of TT} 
There exists an open set $W$ in $\XX$  such 
that $\TT_\omega: W \rightarrow 
C^r(\T\times I_\delta, I_\delta)$ is a well defined 
continuous map. 
Additionally consider $U=\left(p_0 \circ \TT_\omega\right)^{-1} (\MM_\delta)$, then 
we have that there exists an open set $W'$ in 
$U\cap \XX$ such that $\DD_0(\TT) \subset W'\subset \DD(\TT)$. 
\end{prop}

\begin{rem}
\label{Remmark dependency of set on parameters}
Note that the sets $\XX$, $\DD(\TT)$, $\DD_0(\TT)$ 
and the operator $\TT_\omega$ depend 
on the value $\delta$, but here we do not make it explicit
to keep the notation simple. On the other 
hand, we made explicit the dependence of $\TT_\omega$ on 
$\omega$, but not in the set $\DD(\TT)$ and $\DD_0(\TT)$, which a priori 
should also depend on $\omega$. The set  $\DD_0(\TT)$ is 
an isomorphic copy of $\DD(\RR_\delta)$, then it does not depend on 
$\omega$. The set $\DD(\TT)$ actually depends on $\omega$, 
but it is omitted from the notation for simplicity. 
\end{rem}

\begin{rem} 
\label{remark doubling of the rotation number}
Consider a map $F$ like (\ref{q.p. forced system interval}). We have 
that the map is determined by a function $f\in \XX$ and a value 
of $\omega\in \T$. Then we can define the renormalization of $(\omega,f)$ 
as the map determined by $(2\omega, \TT_\omega f)$, which acts 
from $\T \times \DD(\TT)$ to $\T \times \XX$. Note that the frequency 
$\omega$ has been doubled. This is due to the fact that 
the renormalization of a map is constructed
 as the affine transformation of the map iterated twice. 
For convenience, in the remaining of this section we will study 
$\TT_\omega$ as an operator acting from $\DD(\TT)$ to $\XX$ 
depending on the parameter $\omega$. In section \ref{chapter application}
we will take into account the doubling of the 
rotation number again. 
\end{rem}

\subsubsection{Proofs}

\begin{proof}[Proof of proposition \ref{condition to belong  in X}] 
Note that $p_0(g)=p_0(f) + p_0(h) = f\in \MM_\delta$. 
To prove that $g\in \XX$, it is only necessary to check that
$g\in C^r(\T\times I_\delta, I_\delta)$. The map $g$ is $C^r$ 
for being the linear combination of $C^r$ maps. The
function $g$ belongs to  $C^r(\T\times I_\delta, I_\delta)$ 
if $g(\theta,x) \in I_\delta$ for any $\theta \in \T$ and $x \in I_\delta$. 

We begin with the upper bound. As $f(x) \leq f(0) =1 $ 
for any $x\in I_\delta$, then
\[
g(\theta, x) = f(x) + h(\theta,x) \leq f(x) + \|h\| \leq f(0) +\eps = 1+\eps,
\]
for any $(\theta,x) \in  \T\times I_\delta$.

Now we check the lower bound. We have that 
$\gamma$ only depends on $f$ and it is always greater or equal to 0.
On the other hand, we have $f(x) \geq f(1+\delta)$ for any $x\in I_\delta$. Then
\[
g(\theta, x) = f(x) + h(\theta,x) \geq f(x) - \|h\|  \geq 
f(1+\delta) - \|h\| = -1-\delta+\gamma-\|h\|,
\]
which implies that $g(\theta,x) \geq -1-\delta$ if $\|h\|<\gamma$.
\end{proof}


\begin{proof}[Proof of proposition 
\ref{proposition existence of open set domain of TT}]
First of all we need to build an open set where the 
q.p. renormalization operator (\ref{operator tau}) is well 
defined. With this aim we have the following lemma.

\begin{lem}
\label{lemma open domain} 
There exists an open set $U_1\subset \XX$ with $\DD_0(\TT) 
\subset U_1 $ such that the map 
\[
\begin{array}{rccc}
F_1:& U_1\subset C^r(\T \times I_\delta, I_\delta) &\rightarrow & C^r(\T 
\times I_\delta, I_\delta) \\
\displaystyle \rule{0pt}{3ex} & g= g(\theta,x) & \mapsto &  \rule{0pt}{3ex}
\displaystyle [F_1(g)](\theta,x) =  g(\theta,\hat{a} x), 
\end{array} 
\] 
with $\displaystyle \hat{a}=\int_0^1 g(\theta,1) d\theta$, is well defined
and continuous. 
\end{lem} 

\begin{proof}
Given a function $g \in \XX$ we have that $p_0(g)\in \MM_\delta$. 
Note that the value $\hat{a}$ as a functional 
operator $\hat{a}: \C^r(\T\times I_\delta, I_\delta) \rightarrow I_\delta$
is equal to the evaluation map 
at $x=1$ composed with the projection $p_0$. This is a 
bounded linear operator, therefore we have that 
$\hat{a}: C^r(\T\times I_\delta,I_\delta) \rightarrow I_\delta$ is  
continuous.
Then we can consider $U_1= \hat{a}^{-1}((-1,1))$, which 
is an open set because it is the preimage of 
an open set by a continuous function. 
For any $g\in U_1$ using $|\hat{a}|<1$ 
we have  
\[
\sup_{(\theta,x)\in \T\times I_\delta} |g(\theta,\hat{a}x)|\leq
\sup_{(\theta,x)\in \T\times I_\delta} |g(\theta,x)|\leq 1+\delta.
\] Hence $g(\theta,\hat{a}x)$ is well defined 
for any $g\in U_1$ and $(\theta,x)\in \T\times I_\delta$. 
Moreover for any $g\in \DD_0(\TT)$ we 
have $\hat{a}(g)=\left[p_0(g)\right](1) \in (-1,0)$, 
which proves that $\DD_0(\TT)\subset U_1$.  
\end{proof} 

As discussed in the proof of lemma \ref{lemma open domain}
above we have that $\hat{a}: U_1\rightarrow (-1,1)$ defined as 
$\hat{a}(g) = \int_0^1 g(\theta,1)d\theta$ is  a
continuous function.  

Using the results on the smoothness of the composition map from \cite{Irw72} 
and lemma \ref{lemma open domain} above, 
we have that the map
\[
\begin{array}{rccc}
F_2:& U_1\subset C^r(\T \times I_\delta, I_\delta) &\rightarrow & C^r(\T 
\times I_\delta, I_\delta) \\
\displaystyle \rule{0pt}{3ex} & g= g(x,\theta) & \mapsto &  \rule{0pt}{3ex}
\displaystyle [F_2(g)](\theta,x) = g(\theta+\omega, g(\theta,\hat{a} x)), 
\end{array} 
\]
is well defined and is also continuous. 

On the other hand, let $U_2 =p_0^{-1} (\DD_0(\TT))$ and consider 
\[
\begin{array}{rccc}
F_3:& U_2 \subset C^r(\T \times I_\delta, I_\delta) &\rightarrow & C^r(\T, \R ) \\
\displaystyle \rule{0pt}{3ex} & g= g(x,\theta) & \mapsto &  \rule{0pt}{3ex}
\displaystyle [F_3(g)](x) = \frac{1}{\hat{a}} x.
\end{array} 
\]
For any $g\in U_2$ we have $\hat{a}(g)<0$, therefore 
the map $F_3$ is well defined. Indeed, we have that it 
is continuous with respect to $g$. 

Finally note that $\TT_\omega$ is obtained as the composition
$(F_3(g))\circ(F_2(g))$. Using the results from \cite{Irw72} 
we have that $\TT_\omega$ is 
continuous (and it is well defined) as an operator 
$\TT_\omega: U_3 \subset C^r(\T \times I_\delta, I_\delta) \rightarrow
C^r(\T \times I_\delta, \R )$, where $U_3 = U_2 \cap U_1$.
Note that $\DD_0(\TT_\omega)$ is the image by the inclusion of $\DD(\RR_\delta)$. 
Therefore we have $\DD_0(\TT_\omega)\subset U_2$ and, consequently,
$\DD_0(\TT_\omega)\subset U_3$.

Note that the image space of the operator is  $\TT_\omega$ is not 
the desired one. Since $\T \times I_\delta$ is compact, we have 
that the map $N:C^r(\T \times I_\delta, \R )\rightarrow [0,\infty)$ defined 
as $N(g)= \|g \|_{\infty}$ is continuous. Consider the set 
$U_4:=N^{-1}([0,1+\delta))$ which is an open subset of $C^r(\T \times I_\delta, \R )$. 
At this point we can define the set $W$  in the statement of the
proposition as
\[W= \TT_\omega ^{-1}( U_4) \cap U_3.\] 
Using this construction of the set $W$ we have $\TT_\omega: W
\subset C^r(\T \times I_\delta, I_\delta) \rightarrow
C^r(\T \times I_\delta, I_\delta)$. Again we have $\DD_0(\TT_\omega)\subset W$
due to the fact that for any $g\in \DD_0(\TT_\omega)$ we 
have $\|\TT_\omega(g)\|_{\infty}= \|\RR_\delta(g)\|_{\infty}  < 1+\delta$. 
This concludes the proof of the first assertion in the proposition.

We check now the second assertion of the proposition. 
Consider $U=\left(p_0 \circ \TT_\omega\right)^{-1} (\MM_\delta)$ and 
$U_5=W\cap U$. Consider also the following auxiliary functions, 
\[
\begin{array}{rccc}
F_4:& U_5 \subset C^r(\T \times I_\delta, I_\delta) &\rightarrow & [0,+\infty), \\
\displaystyle \rule{0pt}{3ex} & g= g(x,\theta) & \mapsto &  \rule{0pt}{3ex}
\displaystyle \|\TT_\omega(g) - p_0(\TT_\omega(g)) \|.
\end{array} 
\]
and 
\[
\begin{array}{rccc}
F_5:& U_5 \subset C^r(\T \times I_\delta, I_\delta) &\rightarrow & \R, \\
\displaystyle \rule{0pt}{3ex} & g= g(x,\theta) & \mapsto &  \rule{0pt}{3ex}
\displaystyle 1+\delta +  [\TT_\omega(g)](1+\delta) - 
\|\TT_\omega(g) - p_0(\TT_\omega(g)) \|).
\end{array} 
\]

Now we can define the set $W'$  in the statement of the 
proposition as 
\[W':= U_5\cap F_4^{-1}([0,\delta))\cap F_5^{-1}((-\delta,\delta)).\] 

First let us check that $\DD_0(\TT_\omega)\subset W \subset \DD(\TT)$. 
For any map $g\in \DD_0(\TT_\omega)$ we have that 
$\TT_\omega(g) =\RR_\delta(g)= p_0(\RR_\delta(g))= p_0(\TT_\omega(g))$; 
then it follows easily that $\DD_0(\TT_\omega)\subset F_4^{-1}([0,\delta))$. 
Moreover as $\TT_\omega(g) = \RR_\delta(g)\in \MM_\delta$, we have
$[\RR_\delta(g)](1+\delta) < 1+\delta$,  which implies that 
$g\in  F_5^{-1}(-\delta,\delta)$. 

Finally, we have to check  $W \subset \DD(\TT)$, which is equivalent to
prove that $\TT_\omega(g)\in \XX$ for any $g\in W$. Given $g\in W$, 
we have 
$\TT_\omega(g) = p_0(\TT_\omega(g)) + \TT_\omega(g) - p_0(\TT_\omega(g))$. 
From $g\in U_5$ we have that $p_0(\TT_\omega(g))\in \MM_\delta$. Moreover
 from $g\in  F_4^{-1}([0,\delta))$ we have 
$ \|\TT_\omega(g) - p_0(\TT_\omega(g)) \|< \delta$ and from 
 $g\in  F_5^{-1}([0,\delta))$ we have $ \|\TT_\omega(g) - p_0(\TT_\omega(g)) \|< 
1+\delta + \left[p_0(\TT_\omega)\right](1+\delta) $. Note 
that as $p_0\left(\TT_\omega(g) - p_0(\TT_\omega(g))\right)=0$, 
we can apply proposition \ref{condition to belong  in X} and then 
it follows that $g\in \XX$. 
\end{proof}

\subsection{Study of the operator $\TT_\omega$}
\label{subsection study of TT}

Let us follow with the study of 
the operator $\TT_\omega$. In this section we start showing that the 
fixed points of $\RR_\delta$ extend to 
fixed points of $\TT_\omega$. Then we give a result on the 
differentiability of $\TT_\omega$, in the $C^r$ topology. 
With this result it becomes evident that the $C^r$ 
topology is a bad choice for the study of the operator. 
Lastly, we introduce the topology of real analytic 
maps and we check that the operator is 
well defined and differentiable if certain hypothesis (which 
will be called {\bf H0}) is satisfied. Again, all the proofs have been moved to 
the end of the section. 
%
\begin{prop}
\label{proposition fixed points extend}
The operator $\TT_\omega$ restricted to  the set $\DD_0(\TT)$,
is isomorphically conjugate to $\RR_\delta$. Concretely we 
have that any fixed point of $\RR_\delta$ extends 
to a fixed point of $\TT_\omega$. 
\end{prop}

We have the following
result on the differentiability of $\TT_\omega$.

\begin{thm}
\label{theorem renormalization oprator is differentiable}
Let $\TT_\omega:\DD(\TT) \rightarrow \XX$ be the renormalization operator in 
the $C^r$-topology, and consider $\Phi$ a fixed point of 
the operator. If $\Phi\in \DD_0(\TT)\cap
C^{r+s}(\T\times I_\delta,I_\delta)$ then we have that 
there exists $U$ an open neighborhood of $\Phi$ in
$\DD_0(\TT)\cap C^{r+s} (\T\times I_\delta,I_\delta)$
such that $\TT_\omega$ 
is a $C^s$ operator in $U$. 
Moreover, if $s\geq 1$ for any point $\Psi\in U $ we have 
that the Gateaux derivative of
$\TT_\omega$ on $\Psi$ in the direction $h$ is given by
\begin{equation} 
\label{diferencial de TT 0}
\begin{array}{rl}
\displaystyle [d \TT_\omega(\Psi, h)](\theta,x) = &  
\displaystyle  \phantom{+}\frac{1}{a} \psi'(\psi( a x )) h(\theta,
a x)  +  \frac{1}{a} h(\theta + \omega, \psi(a x)) \\ 
\rule{0pt}{4ex} & \displaystyle + \frac{b}{a}  \psi'(\psi(ax))\psi'(ax) x
- \frac{b}{a^2} \psi(\psi(ax)),
\end{array}
\end{equation}
where $\psi=p_0(\Psi)$,  
$a=\psi(1)$ and $b= \displaystyle  
\int_0^{1} h(\theta,1)d\theta$.
\end{thm}

Note that there is a ``loss of differentiability'', in the sense
that one needs to assume that the function $\Psi$ where 
we differentiate the operator is in $C^{r+s}$ while the 
operator acts in subsets of  $C^{r}(\T\times I_\delta, I_\delta)$.
This is due to the self composition in the renormalization operation. 
To skip this problem, let us introduce the topology 
of analytic functions instead of the $C^r$ one, for the 
forthcoming study of the operator. 

\begin{defin} 
\label{definition space BB} 
Let $\W$ be an open set in the complex plane containing the 
interval $I_\delta$ and let 
$\B_\rho = \{z = x + i y\in \C \text{ such that } |y| < \rho\}$.   The 
we define the set $\BB=\BB(\B_\rho,\W)$ as the space of functions 
$f: \B_\rho \times  \W \rightarrow \C$ such that:
\begin{enumerate} 
\item  $f$ is holomorphic in $\B_{\rho}\times \W$ and continuous 
in the closure of  $\B_{\rho}\times \W$.

\item $f$ is real analytic (it maps real numbers to real numbers). 

\item  $f$ is $1$-periodic in the first variable, i. e.  
$f(\theta +1,z) = f(\theta,z)$ for any $(\theta,z) \in \B_{\rho}
\times\W$. 
\end{enumerate}
\end{defin}

\begin{prop} 
\label{proposition BB is banach}
The space $\BB$ endowed with the supremum norm  
$\left( \displaystyle \| f\| = \sup_{\B_{\rho_1}\times\W}
| f(\theta,z)| \right)$  is  a Banach space. 
\end{prop}

We want to consider the quasi-periodic renormalization 
operator $\TT_\omega$ (see equation (\ref{operator tau})) 
restricted to the domain 
$ \DD(\TT) \cap \BB$, but then it is necessary to check that 
$\TT_\omega$ is well defined in the complex domain. 
For any function $f\in  \DD(\TT) \cap \BB$ 
we should check that $f( \B_{\rho}\times a \W) 
\subset \W$ (where $a\W=\{z\in \C |
\thinspace a z\in \W \}  $). In the one 
dimensional renormalization theory  the open set $\W$ 
is chosen such that this condition is satisfied. Concretely, 
this condition is 
typically checked with computer assistance, together with 
other conditions to prove the existence of fixed points (see 
\cite{Lan82,Lan92,KSW96}).
In our case we will assume that the following hypothesis is satisfied.

\begin{description}
\item[H0)] There exists an open set $\W\subset \C$ containing 
$I_\delta$ and a function $\Phi\in \BB \cap \XX_0$ such 
that $p_0(\Phi)$ is a fixed point of the  
renormalization operator $\RR_\delta$  and such that 
the closure of both $a\W$ and $p_0(\Phi)(a \W)$ is 
contained in $\W$ (where $a:=\Phi(1)$)
\end{description}

Although the results on the existence of the fixed point of 
renormalization operator done by Lanford in \cite{Lan82} are well 
accepted, in the cited article there are no proofs and many details are omitted. 
For the proof of the existence of the fixed point 
it is also necessary to check this condition. 
In \cite{Lan92}, some more details on how to prove the existence of 
the fixed point are given. Actually, 
it is claimed that the hypothesis {\bf H0} is true for the set 
\[
\left\{ z\in \C \text { such that } |z^2 -1| < \frac{5}{2} \right\}.
\] 

This set used by Lanford is more convenient in his study 
since he works in the 
set of even holomorphic functions. In the numerical computations 
from \cite{JRT11c} we use as $\W$ the disc centered at 
$\frac{1}{5}$ with radius $\frac{3}{2}$, 
and we check hypothesis {\bf H0} numerically 
(without rigorous bounds).

\begin{thm}
\label{theorem renormalization analitic set up}
Assume that {\bf H0} holds and let $\Phi$ be 
the fixed point given by this assumption. Then we have that there exists 
$U\subset \DD(\TT)\cap\BB$, an open neighborhood of $\Phi$, 
such that $\TT_\omega :U\rightarrow \BB$ is well defined, 
and $\TT_\omega$ is 
Fr\'echet differentiable for any  $\Psi \in U$ and the derivative is equal to 
\begin{equation} 
\label{diferencial de TT}
\begin{array}{rl}
\displaystyle [D \TT_\omega(\Psi, h)](\theta,x) = &  
\displaystyle  \phantom{+}\frac{1}{a} (\partial_x\Psi)
(\theta+\omega, \Psi(\theta,a x )) h(\theta, a x)  +  
\frac{1}{a} h(\theta + \omega, \Psi(\theta, a x)) \\ 
\rule{0pt}{4ex} & \displaystyle + \frac{b}{a}  (\partial_x \Psi)
(\theta+\omega,\Psi(\theta,ax))
(\partial_x \Psi (\theta,ax)) x - \frac{b}{a^2} \Psi(\theta+\omega,\Psi(\theta,ax)),
\end{array}
\end{equation}
where $a=\displaystyle \int_0^{1}\Psi(\theta,1)d\theta$ and $b= \displaystyle  
\int_0^{1} h(\theta,1)d\theta$.
\end{thm}

\subsubsection{Proofs}

\begin{proof}[Proof of proposition \ref{proposition fixed points extend}]
If a map belongs to $\DD_0(\TT)$ then it does not 
depend on $\theta$, therefore the operator $\TT_\omega$ coincides
with $\RR_\delta$ composed with the inclusion of $\MM_\delta$ in $\XX$. 
\end{proof} 

\begin{proof}[Proof of theorem 
\ref{theorem renormalization oprator is differentiable}]
By definition (see equation (\ref{operator tau})) given a function
$g\in  C^r(\T \times I_\delta, I_\delta)$ we have that
$ [\TT_\omega(g)](\theta,x) :=  \displaystyle\frac{1}{\hat{a}} g(\theta + \omega,
g(\theta, \hat{a}x))$ where $\displaystyle \hat{a} =
\int_{0}^{1} 
g(\theta, 1) d\theta$. Note that $\hat{a}$ can be written as $\hat{a}= [p_0(g)](1)$. The function
$p_0:  C^{r+s}(\T \times I_\delta, I_\delta) \rightarrow  C^{r+s}(I_\delta, I_\delta)$  
defined by (\ref{equation projection 0}) is $C^s$
(actually it is a linear bounded operator). On the other hand the evaluation 
of a $C^r$ function in a concrete
value is also a $C^r$ function 
(see proposition 2.4.17 from \cite{AMR88}). 
Therefore $\hat{a}=\hat{a}(g)$ as a function of $g$  is $C^r$ as well.

Note that $\TT_\omega(g)$ 
can be written as the composition of
different functions, which are $g$ itself,
a translation in the $\theta$ variable and a scalar multiplication by $a$ (and its inverse) in the $x$
variable. Each one of these functions are $C^r$ dependent with respect to $f$
except the composition of $f$ with itself which is only a $C^s$ map, 
when we work in the $C^r$ topology 
(see \cite{Irw72}). We can conclude that $\TT_\omega$ is only a $C^s$ operator.

Now we compute explicitly the Gateaux
derivative. As $\Psi$ belongs to $\XX_0$, we have $\Psi(x,\theta)=\psi(x)$ 
and consequently
\begin{eqnarray}
\label{operator derivative computation}
\displaystyle
\left[\TT(\Psi + t h)\right] (\theta, x) & = &
\displaystyle
\frac{1}{\hat{a}} \left[ \Psi + t h \right] ( \theta +\omega,
\left[ \Psi + t h \right](\theta,\hat{a} x)) \nonumber  \\
& = &  \displaystyle  \frac{1}{\hat{a}} 
\psi (\psi( \hat{a}x)) +  
\frac{1}{\hat{a}} \psi'(\psi(\hat{a}x )) h(\theta,\hat{a}x) t \nonumber
 \\ & &  \displaystyle 
+ \frac{1}{\hat{a}}   h(\theta + \omega, \psi(\hat{a}x)) t  + O(t^2),
\end{eqnarray}
where $\displaystyle \hat{a} =
\int_{0}^{1} [\Psi + t h](1,\theta) d\theta$.

Set $a= \psi(1)$ and $\displaystyle b=\int_{0}^{1} 
h(1,\theta) d\theta$. We have that $\hat{a}= a+ tb$.
Therefore,
\begin{equation}
\label{Taylor 1 ente a}
\frac{1}{\hat a} = \frac{1}{a+tb}= \frac{1}{a} - \frac{1}{a^2} t b 
+ O(t^2),
\end{equation}
and using the chain rule we have
\begin{equation}
\label{psiopsi de hat ax}
\psi\left(\psi( \hat{a} x) \right) = \psi\left(\psi( a x) \right)  
 + \psi'\left(\psi(ax)\right) \psi'(ax) t   b x  + O(t^2). 
\end{equation}

Combining equations (\ref{operator derivative computation}),
(\ref{Taylor 1 ente a}) and (\ref{psiopsi de hat ax}) follows that
the Gateaux derivate of $\TT$ in $\Psi$ is the one given by
(\ref{diferencial de TT}).
\end{proof}

\begin{proof}[Proof of proposition \ref{proposition BB is banach}] 
Consider $\AAA$ the space of 
holomorphic functions in $\B_{\rho}\times \W$ and continuous
in the closure of  $\B_{\rho}\times \W$.
Using basic properties of the holomorphic functions 
in several variables (see \cite{Hor73, Kra82}) is 
easy to check that $\AAA$ is a Banach space. The space $\BB$ is the set 
of functions of $\AAA$ such that 
\begin{itemize}
\item $f(\theta + 1 ,z) - f(\theta, z) =0 $ for any 
$(\theta,z) $ in $\B_{\rho} \times\W$.
\item $f(\theta,x) - \overline{f(\theta,x)} = 0$ for any 
$(\theta,x)$ in $\R \times I_\delta$. 
\end{itemize}
Then $\BB$ is the preimage of a closed subset 
by a continuous function, therefore it is closed in $\AAA$ 
and consequently it is a Banach space.  
\end{proof}


\begin{proof}[Proof of the theorem 
\ref{theorem renormalization analitic set up}]
Given a function $f\in  \DD(\TT) \cap \BB$, we have 
that its image by $\TT_\omega(f)$ is given by (\ref{operator tau}). 
If we want it to be well defined we must 
check that $f( \B_{\rho}\times a \W) 
\subset \W$ (where $a\W=\{z\in \C |
\thinspace a z\in \W \}  $).

We have that $\phi(\theta,x) = \left[p_0(\phi)\right] (x)$ for any $\theta\in \B_{\rho}$. 
Using hypothesis {\bf H0}  we have that 
$\Cl\left(\Phi\left(\B_{\rho}\times a \W\right)\right)\subset \W$, 
where $\Cl(\cdot)$ denotes the closure of the set. 
If we consider now a function $f$  in a 
suitable neighborhood 
of $\Phi$ we have that it still maps $ \B_{\rho} 
\times a \W$ inside $\W$ (if 
$f$ is close enough to $\Phi$ in the topology of $\BB$).

To prove the differentiability of the map we will check it 
directly from the definition of Fr\'echet derivative.

From $\Cl\left(\Phi\left(\B_{\rho}\times a \W\right)\right)\subset \W$, 
and the fact of $\W$ being bounded it follows that 
$\Cl\left(\Phi\left(\B_{\rho}\times a \W\right)\right)$ is compact. 
Consider the following filtration of sets in the 
complex plane
\[
\Cl\left(\Phi\left(\B_{\rho}\times a \W\right)\right) = K_0 \subset V_0 \subset 
K_1 \subset V_1 \subset  K_2 \subset V_2= \W,
\]
with each $K_i$ compact and each $V_i$ open, for $i=0,1,2$. 

Consider now $U_1\subset U$ the open neighborhood of $\Phi$ in $\BB$ formed 
by the $\Psi\in \BB$  such that 
\[
\Psi\left(\B_{\rho}\times a \W\right) \subset V_0. 
\]
For any map $\Psi\in U_1$ we have that 
$\Cl\left(\Psi\left(\B_{\rho}\times \hat{a} \W\right)\right)\subset K_1$. 

On the other hand, from  $K_2 \subset \W$ and the fact 
that $K_2$ is compact and $\W$ open, it follows that 
there exist a value $r>0$ such that for any $x_0 \in K_2$ 
the ball centered on $x_0$ with radius $r$ is 
contained in $\W$. Then for any map $f\in \BB$ we have 
\[
\partial_x f(\theta,x_0) =\frac{1}{2\pi i} \int_{|z-x_0|=r} 
\frac{f(\theta,z)}{(z-z_0)^2}  d\theta. 
\]
Then it follows easily that, for any $f \in \BB$ and 
$x_0\in K_2$ we have 
\[
| \partial_x f(\theta,x_0)| \leq \frac{1}{r} \| f \|_\infty.
\]
Modifying  the same argument, we can check that 
\[
| \partial^2_{x^2} f(\theta,x_0)| \leq \frac{2}{r^2} \| f \|_\infty.
\]
Note that both bounds are uniform for any map in $U_1$.

Consider $\Psi \in U_1$, and $h \in \BB$ with $\|h\|$ small. We 
want to compute $\TT_\omega(\Psi+h)$ up to $O(\|h\|^2)$. 
First of all we have, 
\begin{equation}
\label{equation analicity 0.0}
\TT_\omega(\Psi+h)= 
\frac{1}{\hat{a}(\Psi+h)} \left[ 
\begin{array}{r}
\Psi(\theta+\omega, \Psi(\theta,\hat{a}(\Psi+h)x) + h(\theta,\hat{a}(\Psi+h)x))  + \\
h(\theta+\omega, \Psi(\theta,\hat{a}(\Psi+h)x) + h(\theta,\hat{a}(\Psi+h)x)) 
\end{array}
\right].
\end{equation}

To simplify the notation consider  
\[
a= \int_0^1 \Psi(\theta,1)d\theta , \quad 
b= \int_0^1 h(\theta,1)d\theta. 
\]
Then we have $\hat{a}(\Psi+h) = a + b$, and
\[
|b| \leq  \int_0^1 |h(\theta,1)|d\theta \leq \|h \|.
\]

Since $\Psi \in U_1$ we have that for any $h$ with $\|h\|$ sufficiently 
small $\Psi + h \in U_1$, therefore we have that 
$\Psi(\theta,(a+b)x) + h(\theta,(a+b)x) 
\in V_1$. Using the complex Taylor expansion with respect to $x$ up 
to second order we have
\begin{eqnarray}
\label{equation analicity 1.1}
\Psi(\theta+\omega, \Psi(\theta,(a+b)x) + h(\theta,(a+b)x)) &= & 
\Psi(\theta+\omega,\Psi(\theta,(a+b)x)) +
  \\ & &
(\partial_x \Psi)(\theta+\omega,\Psi(\theta,(a+b)x))  h(\theta,(a+b)x)) + \nonumber \\
& &  R_2(\theta,x) \nonumber 
\end{eqnarray}

with 
\begin{equation}
\label{equation analicity 1.01}
| R_2(\theta,x) | \leq \frac{1}{r^2}\frac{1}{1-\frac{\|h\|}{r}} \| \Psi\| \|h\|^2 =
O(\|h\|^2).
\end{equation} 

Analogously we have 
\begin{eqnarray}    
\label{equation analicity 1.2}
h(\theta+\omega, \Psi(\theta,(a+b)x) + h(\theta,(a+b)x)) &= & 
h(\theta+\omega,\Psi(\theta,(a+b)x)) +   \\ & &
R_1(\theta,x), \nonumber
\end{eqnarray}
with 
\begin{equation} 
\label{equation analicity 2.0}
| R_1(\theta,x) | \leq  \frac{1}{r}\frac{1}{1-\frac{\|h\|}{r}} \| h\| \|h\| = O(\|h\|^2).
\end{equation} 

Recall that $|b|=O(\|h\|)$  then applying Taylor expansion 
and the uniform bound on $K_2$ it follows easily that 
\begin{eqnarray}
\label{equation analicity 3.1}
\Psi(\theta,(a+b)x) &= &  \Psi(\theta,ax) + (\partial_x \Psi) 
(\theta,ax) b x + O(\|h\|^2),  \\
\label{equation analicity 3.2}
h(\theta,(a+b)x) &= &  h(\theta,ax)  +  O(\|h\|^2).  
\end{eqnarray}

Using that $\Psi\in U_1$  we have that
$ \Psi(\theta,ax) + (\partial_x \Psi) (\theta,ax) b x$
belongs to $V_1\subset K_2$ for $\|h\|$ sufficiently small. Now 
we can combine this fact with the uniform bound on $K_2$  and 
equation (\ref{equation analicity 3.1}) and (\ref{equation analicity 3.2}) to prove that  
\begin{equation}
\nonumber 
h(\theta+\omega,\Psi(\theta,(a+b)x)) = h(\theta+\omega,\Psi(\theta,ax)) + O(\|h\|^2). 
\end{equation}
Using this together with equations (\ref{equation analicity 1.2}) and 
(\ref{equation analicity 2.0}) we obtain
\begin{eqnarray}    
\label{equation analicity 4.2}
h(\theta+\omega, \Psi(\theta,(a+b)x) + h(\theta,(a+b)x)) &= & 
h(\theta+\omega,\Psi(\theta,a x)) +  O(\|  h\|^2) . 
\end{eqnarray}

With the same argument it follows that
\[
(\partial_x \Psi)(\theta+\omega,\Psi(\theta,(a+b)x)) = 
(\partial_x \Psi)(\theta+\omega,\Psi(\theta,ax)) + O(\|h\|) ,
 \]
and 
\begin{eqnarray}
\Psi(\theta+\omega,\Psi(\theta,(a+b)x)) &=& 
\Psi(\theta+\omega,\Psi(\theta,ax)) + \nonumber \\ 
& & (\partial_x\Psi) (\theta+\omega,\Psi(\theta,ax)) (\partial_x \Psi)(\theta,ax) bx \nonumber 
+ 
 O(\|h\|^2).
\end{eqnarray}
Replacing the last two equations in equation (\ref{equation analicity 1.1}) and using the 
bound given by (\ref{equation analicity 1.01}) yields to 
\begin{eqnarray}
\label{equation analicity 4.1}
\Psi(\theta+\omega, \Psi(\theta,(a+b)x) + h(\theta,(a+b)x)) &= & 
\Psi(\theta+\omega,\Psi(\theta,ax)) +
  \\ & &
(\partial_x \Psi) (\theta+\omega,\Psi(\theta,ax)) 
(\partial_x \Psi)(\theta,ax) bx \nonumber + 
  \\ & &
(\partial_x \Psi)(\theta+\omega,\Psi(\theta,ax))  h(\theta,(a+b)x)) 
+ O(\|h\|^2). \nonumber 
\end{eqnarray}

Finally, recall that $|b|= O(\|h\|)$, therefore
\[
\frac{1}{a+b}= \frac{1}{a} - \frac{b}{a^2} + O(\|h\|).
\]
When we replace this value and the ones of 
equations (\ref{equation analicity 4.2}) and (\ref{equation analicity 4.1}) 
in (\ref{equation analicity 0.0}) it follows that 
\[\left\|\TT_\omega(\Psi+h) - \TT_\omega(\Psi) - D\TT_\omega(\Psi)h \right\|= 
O\left(\|h\|^2\right),\]
which proves the differentiability of the operator in the analytic context. 
\end{proof}

\subsection{Fourier expansion of $D\TT_\omega(\Psi)$.} 
\label{section The Fourier expansion of DT} 

Let $\Psi$ be a function as in the hypothesis of theorem 
\ref{theorem renormalization analitic set up}, but additionally 
assume that $\Psi\in U\cap \DD_0(\TT_\omega)$. 
In this section we study $D\TT_\omega(\Psi)$, the 
differential of the quasi-periodic renormalization operator. Concretely, 
given $f\in \BB$ we study the Fourier expansion of $D\TT_\omega(\Psi) f$ 
in terms of the Fourier expansion of $ f$. 
It will turns out that, fixed a positive 
integer $k$, the spaces generated by functions of the type 
$f(x) \cos(2\pi k \theta) + g(x) \sin(2 \pi k \theta)$ (with $f$ and 
$g$ one dimensional real analytic functions) are invariant by  
$D\TT_\omega(\Psi)$. This allows us to reduce the study of 
$D\TT_\omega(\Psi)$ to a simpler operator $\LL_\omega$. We finish 
giving different spectral properties on the operator $\LL_\omega$. 
As usual the proofs have been moved to the end of the 
section.

Given a function $f\in \BB$ we can consider its complex Fourier 
expansion in the periodic variable 
\begin{equation}
\label{fourier expansion complex}
f(\theta,z)= \sum_{k\in \Z} c_k(z) e^{2\pi k\theta i }, 
\end{equation}
where 
\[
c_k(z)= \int_{0}^{1} f(\theta,z) e^{-2\pi k\theta i} d\theta.
\]
We can also consider its real Fourier expansion
\begin{equation}
\label{fourier expansion real}
f(\theta,z)= a_0(z) + \sum_{k>0} a_k(z) \cos(2\pi k\theta) +  b_k(z) \sin(2\pi k\theta).
\end{equation}
Here the coefficients are given as, 
\[
a_0(z)= \int_{0}^{1} f(\theta,z) d\theta, 
\]
\[
a_k(z)=  \frac{1}{2} \int_{0}^{1} f(\theta,z) \cos(2\pi k\theta)  d\theta,  \quad k> 0,
\]
\[
b_k(z)= \frac{1}{2} \int_{0}^{1} f(\theta,z) \sin(2\pi k\theta)  d\theta,  \quad k>0.
\]
We have the well known relation  
between both expansions, given by
$c_k(z) = \displaystyle \frac{a_k(z) + i b_k(z)}{2}$ when $k>0$ and 
$c_0(z) = a_0(z)$. 

Note that each function $c_k$ is holomorphic in $\W$
but not real holomorphic (the image of a real number will not be a real number
in general). On the other hand the real Fourier coefficient $a_k(z)$  
and $b_k(z)$ are real holomorphic in $\W$. 


Let $\psi=p_0(\Psi)$ be representative of $\Psi$ as 
a one dimensional map. Evaluating  (\ref{diferencial de TT}) 
on $c(z) e^{2\pi k \theta i}$ (for any $k\neq0$) we have 
\begin{equation}
\label{invariancia node}
[D\TT_\omega(\Psi)] \left(c_k(z) e^{2\pi k\theta i} \right)= 
\frac{1}{a} \left([ \psi'\circ\psi]( a z ) c_k(az) 
+  [c_k \circ \psi](a z) e^{2\pi k \omega i} \right) e^{2\pi k \theta i }. 
\end{equation}

On the other hand, when (\ref{diferencial de TT}) is 
 evaluated at $c_0(z)$ one has
\[[D\TT(\Psi)] \left(c_0(z) \right)= D\RR_\delta(\psi) c_0(z),\]
as should be expected since $c_0 = p_0(f)$.


Given $U$ an open subset of $\C$ we will denote by $\RHH(U)$ the set of
real holomorphic functions in $U$ and continuous in its closure. 

Consider the operators, 
\[
\begin{array}{rccc}
L_1: & \RHH(\W) & \rightarrow & \RHH(\W)  \\
  &   g(z)  & \mapsto & 
\displaystyle  \frac{1}{a} \psi'\circ\psi(a z) g(az),
\end{array} 
\]
and 
\[
\begin{array}{rccc}
L_2: & \RHH(\W) & \rightarrow & \RHH(\W)  \\   &   g(z)  & \mapsto &
\displaystyle  \frac{1}{a} g\circ\psi(a z),
\end{array}
\] with $\psi=p_0(\Psi)$ and 
$a=\psi(1)$. 


Given a function $f\in \BB$, we can consider its Fourier expansion 
(\ref{fourier expansion complex}) and apply (\ref{invariancia node}), hence
\begin{equation}
\label{equation Fourier expansion differential renormalization operator}
\left[D\TT_\omega(\Psi) f\right](\theta,z) = D\RR_\delta[c_0](z) 
+  \sum_{k\in \Z\setminus\{0\}} \left([L_1(c_k)](z) + 
[L_2(c_k)](z) e^{2\pi k\omega i}\right) e^{2\pi k\theta i }.
\end{equation}
Looking at this formula it can be said that $D\TT_\omega$ ``diagonalizes''
with respect to the complex Fourier base.

We define 
\[
U_k:= \big\{ f \in B |\text{ } f(\theta, x) = u(x) \cos(2\pi k \theta), \text{ for 
some } u\in \RHH(\W)\big\},
\]
and 
\[
V_k:= \big\{ f \in B |\text{ } f(\theta, x) = v(x) \sin(2\pi k \theta), \text{ for
some } v\in \RHH(\W)\big\}.
\]

On the other hand, given $\omega \in \T$, consider the following operator
\begin{equation}
\label{equation maps L_omega}
\begin{array}{rccc}
\LL_\omega: &  \RHH(\W)\oplus  \RHH(\W) & \rightarrow 
&  \RHH(\W)\oplus  \RHH(\W)  \\ \\
& \left( \begin{array}{c} u \\v \end{array} \right) &
\mapsto &  \left( \begin{array}{c}  L_1(u) \\ L_1(v)
\end{array} \right) +
\left( \begin{array}{cc}  \cos(2\pi \omega) & - \sin(2\pi \omega)  \\
\sin(2\pi \omega) & \cos( 2\pi \omega) 
\end{array} \right)  \left( \begin{array}{c}  L_2(u) \\
L_2(v) 
\end{array} \right) . 
\end{array}
\end{equation}

Then we have the following result.

\begin{prop}
\label{proposition invariance of Fourier node spaces} 
The spaces $U_k \oplus V_k$ are
invariant by $D\TT(\Psi)$ for any $k>0$. Moreover $D\TT_\omega (\Psi)$ restricted to $U_k \oplus V_k$ is 
conjugated to $\LL_{k\omega}$. 
\end{prop}

Due to this proposition we have that the understanding of the 
derivative of the renormalization operator in $\BB$ is equivalent 
to the study of the operator $\LL_\omega$ for any $\omega\in \T$. 
Therefore we focus now on the study of  $\LL_\omega$. 

Given a value $\gamma \in \T$, consider the rotation $R_{\gamma}$ defined
as
\begin{equation}
\label{equation rotation rgamma}
\begin{array}{rccc}
R_\gamma: &  \RHH(\W)\oplus  \RHH(\W) & \rightarrow
&  \RHH(\W)\oplus  \RHH(\W)  \\ \\
& \left( \begin{array}{c} u \\v \end{array} \right) &
\mapsto &
\left( \begin{array}{cc}  \cos( 2\pi \gamma) & - \sin(2\pi \gamma)  \\
\sin(2\pi \gamma) & \cos(2\pi \gamma)
\end{array} \right)  \left( \begin{array}{c} u \\ v
\end{array} \right). 
\end{array}
\end{equation}
Then we have the following result.

\begin{prop} 
\label{proposition commutation of LLomega and Rgamma} 
For any $\omega, \gamma \in \T$ we have that $\LL_\omega$ and
$R_\gamma$ commute. 
\end{prop}

This proposition has the following spectral consequences on $\LL_\omega$.  
\begin{cor}
\label{corollary rotational simetry eigenvectors} 
For any eigenvector $(u,v)$ of $\LL_\omega$ we have that $R_\gamma (u,v)$ 
is also an eigenvector of the same eigenvalue for any $\gamma \in \T$.
\end{cor}

\begin{cor}
\label{corollary multiplicity eigenvalues} 
All the eigenvalues of $\LL_\omega$ (different from zero) are either real with 
geometric multiplicity even, or a pair of complex conjugate eigenvalues. 
\end{cor}

On the other hand we have the following result on the dependence 
of the operator with respect to $\omega$

\begin{prop} 
\label{analytically denpendence Lomega}
The operator $\LL_\omega$ depends analytically 
on $\omega$.
\end{prop}

This result allows us to apply theorems 
III-6.17 and VII-1.7 of \cite{Kat66}. These results imply that, as long as 
the eigenvalues of $\LL_\omega$ do not cross each other, then the eigenvalues 
and their associated eigenspaces depend analytically on the parameter $\omega$. 

We want to prove the compactness of $\LL_\omega$ as an operator. 
For technical reasons this can not be proved with $\LL_\omega$ as an 
operator on $\RHH(\W)\oplus\RHH(\W)$, but it can be 
proved on a closed subspace of $\RHH(\W)\oplus\RHH(\W)$.

\begin{prop}
\label{proposition L_omega is compact}
Consider $K\subset \W$ a compact set in the complex plane, such that 
$\psi(a\W)\subset K$ and $a\W \subset K$
where $\psi= p_0 \left( \Psi\right)$. Let us denote by 
$B= \RHH(\W)\cap C^0(K,\C)$, which is a 
Banach subspace of $\RHH(\W)$. 

Then the operator $\LL_\omega$ restricted to the  subspace
$B\oplus B \subset \RHH(\W)\oplus\RHH(\W)$ 
is well defined (i.e. $\LL_\omega:B\oplus B\rightarrow B\oplus B$) 
and it is compact. 
\end{prop}

Recall that the compacity of an operator implies that its 
spectrum is either finite or countable with $0$ on its 
closure (see for instance theorem III-6.26 of \cite{Kat66}). 

\begin{figure}[t]
\begin{center}
\includegraphics[width=10cm]{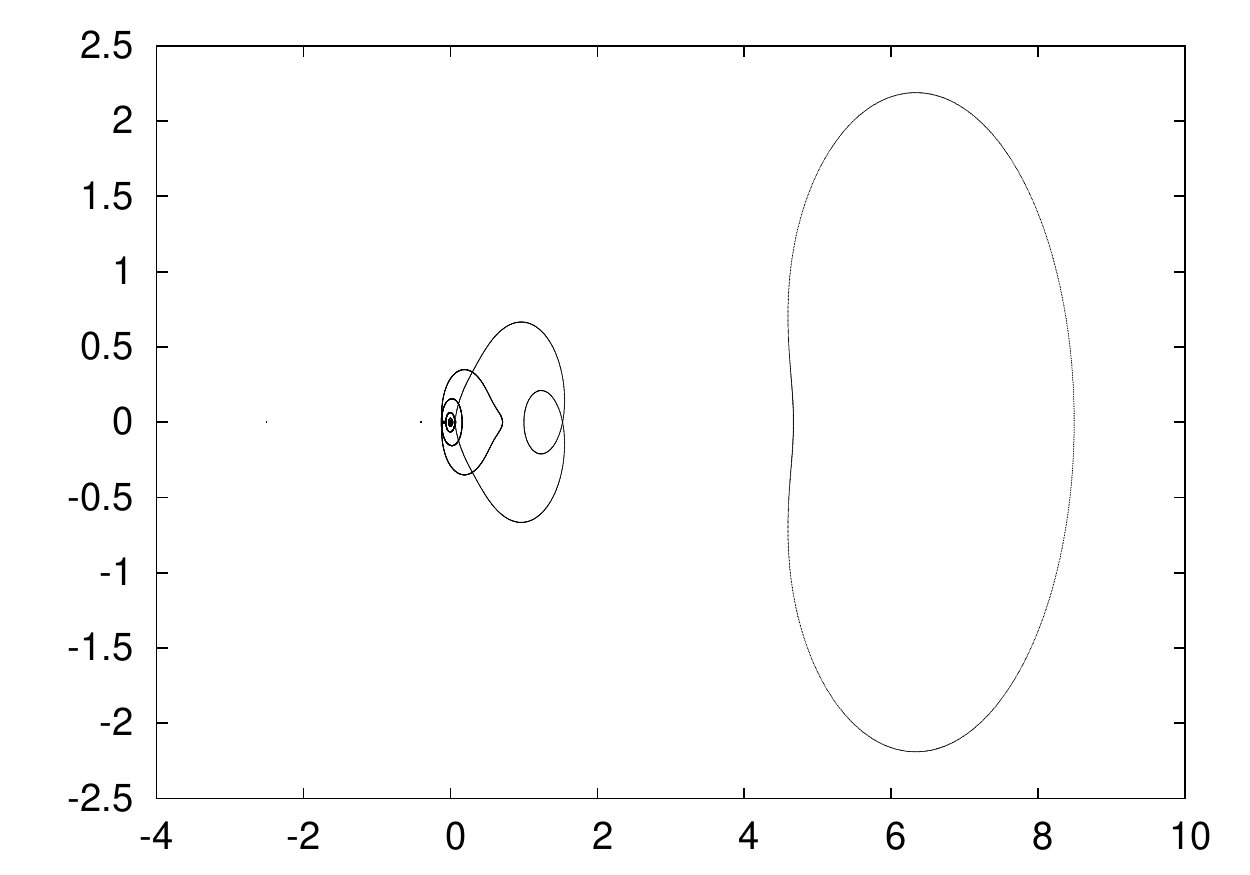}
\includegraphics[width=7.5cm]{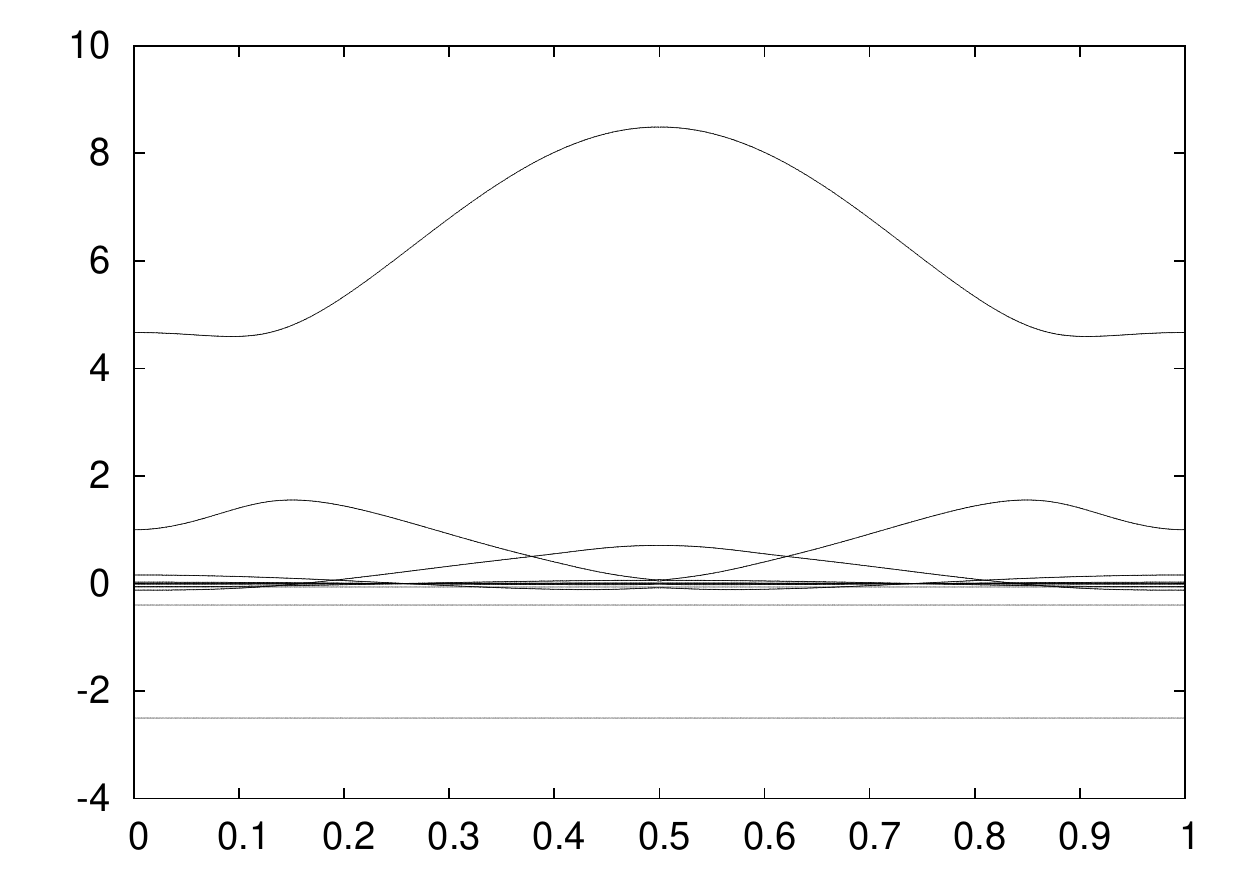}
\includegraphics[width=7.5cm]{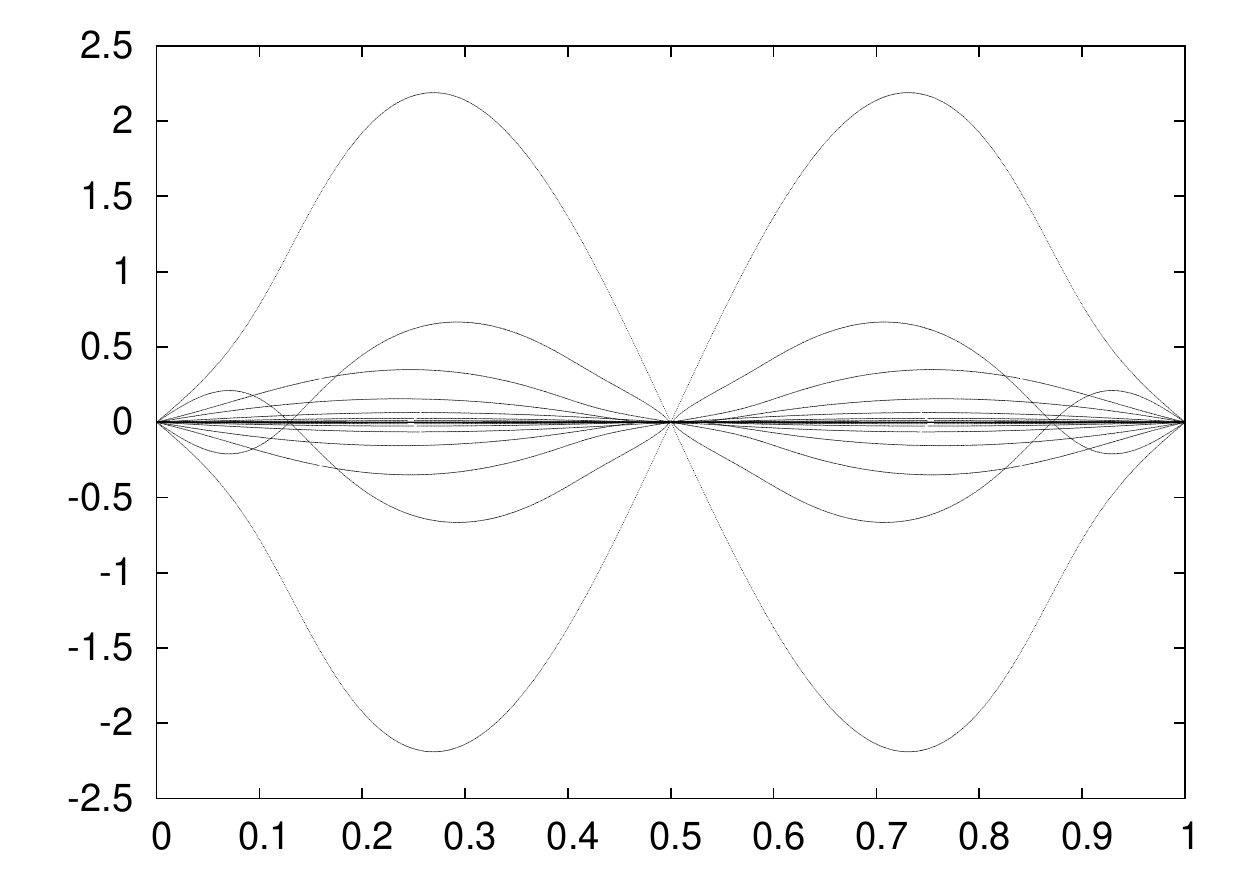}
\caption{ Numerical approximation of the
spectrum of $\LL_\omega$ for $\omega \in \T$.
On the top we have the projection in the complex plane of the spectrum when
$\omega$ varies in $\T$. In the bottom we have the
evolution of the real (left) and the imaginary (right)
part of the spectrum with respect to $\omega$.
}
\label{Operator Spectrum}
\end{center}
\end{figure}

To finish this section we have included in figure 
\ref{Operator Spectrum} a numerical approximation of 
the spectrum of the operator $\LL_\omega$ depending on $\omega$. 
In the figure it can be observed that the different properties
(and their spectral consequences) on the operator stated above are satisfied. 
The details on the numerical computations involved to approximate 
the spectrum are described in \cite{JRT11c}. Several numerical tests on 
the reliability of the results are also included there.

\subsubsection{Proofs}

\begin{proof}[Proof of proposition 
\ref{proposition invariance of Fourier node spaces}]
Let $f$ be a function in $U_k \oplus V_k$, then we have 
$f(\theta,x) = u(z) \cos(2\pi k \theta) + v(z) \sin (2\pi k\theta)$. Consider the 
function $c(z) = \frac{u(z) + i v(z)}{2}$. Using formula (\ref{invariancia node})
on the function $u(z) = c(z) + \bar{c}(z)$ and doing some algebra one can see that 
\[
\begin{array}{rcl}
[D\TT_\omega(\Psi)](u(z) \cos(2\pi k \theta)) &  = & \phantom{-}[L_1(u)](z)\cos(2\pi k \theta) + 
 [L_2(u)](z) \cos(2\pi k \omega) \cos(2\pi k \theta) 
\\ \rule{0pt}{3ex} & &  - [L_2(v)](z)  \sin(2\pi k \omega)  \sin(2\pi k \theta), 
\end{array} 
\]
and, doing a similar calculation for $v(z) = i (\overline{c(z)} - c(z))$
\[
\begin{array}{rcl}
[D\TT_\omega(\Psi)](v(z) \sin(2\pi k \theta))  &  = 
& \phantom{+} [L_1(v)](z)\sin(2\pi k \theta)  + [L_2(u)](z)\sin(2\pi k \omega) \cos(2 \pi k \theta)  
\\ \rule{0pt}{3ex} & &  + [L_2(v)](z) \cos(2\pi k \omega) \sin(2\pi k \theta). 
\end{array}
\]

Now notice that there is a natural isomorphism between $U_k$ and 
$ \RHH(\W)$ given by the function 
$p_c: \RHH(\W) \rightarrow U_k $ defined as 
$p_c(f)(x):= \int_0^{1} f(\theta,x) \cos(2\pi k\theta) d\theta $ and 
the function $i_c:  U_k \rightarrow \RHH(\W)$ defined as 
$i_c(f)(\theta,x) = f(x) \cos(2\pi k\theta)$. The same argument can 
be applied to $V_k$ considering the functions 
$p_s:V_k \rightarrow \RHH(\W)$  and $i_s:\RHH(\W) \rightarrow V_k $ 
defined as before but with $\sin(2\pi k \theta)$ instead of 
$\cos(2\pi k \theta)$. 
Then these functions can be used to define the isomorphic conjugacy
between $D\TT_\omega(\Psi)$ and $\LL_{k\omega}$. 
\end{proof}
\begin{proof}[Proof of proposition
\ref{proposition commutation of LLomega and Rgamma}]
It follows from $L_1$ and $L_2$ being linear and the fact that any pair of
rotations commute.
\end{proof}

\begin{proof}[Proof of corollary 
\ref{corollary rotational simetry eigenvectors}]
Suppose that $(u,v)$ is an eigenvector of eigenvalue $\lambda$, we have 
$\lambda (u,v) = \LL_\omega (u,v)$. Composing in both parts by $R_\gamma$ and 
using the last proposition the result follows. 
\end{proof}

\begin{proof}[Proof of corollary
\ref{corollary multiplicity eigenvalues}]
In the case of a real eigenvalue, suppose it has geometric multiplicity odd, then
its eigenspace is generated by $n$ vectors $y_1, y_2, \dots, y_n$, with $n$ odd. 
We can consider $R_\gamma y_i$ for any $i$, which will also be in the eigenspace of the eigenvalue. 
Since the vector $R_\gamma y_i$ is linearly independent with $y_i$ but it 
is in the eigenspace, we have that it is generated by the other eigenvectors. 
Then one of the original vectors can be replaced by $R_\gamma y_i$. Rearranging the 
vectors if necessary we can suppose that $y_2=R_\gamma y_1$. Doing this 
process repeatedly we will end up with an even number of vectors.

In the case of a complex eigenvalue, using that the operator $\LL_\omega$ 
is real, if it has a complex eigenvalue $\lambda$ with 
eigenvector $v_r + i v_i$, then $\bar{\lambda}$ will also be an eigenvalue with 
eigenvector  $v_r - i v_i$. Given a complex pair of conjugate eigenvalues, the 
restriction of the operator to the corresponding eigenspace can be written as
a uniform scaling composed with a rotation. It can happen that the space 
generated by these two vectors is invariant by the rotation
$R_\gamma$ introduced before. Then the 
multiplicity of the eigenvalue can be odd. Actually, if the pair of complex
eigenvalues are simple, then there exists a $\gamma_0 \in \T$ such that the 
rotation associated to the pair of eigenvectors is $R_{\gamma_0}$. 
\end{proof}

\begin{proof}[Proof of proposition
\ref{analytically denpendence Lomega}]
It follows from the fact that $\LL_\omega$ is the
sum of two bounded linear operators (which do not depend on $\omega$)
times an entire function on $\omega$.
\end{proof}

\begin{proof}[Proof of proposition
\ref{proposition L_omega is compact}]
Note that it is enough to prove that the operators 
$L_1$ and $L_2$ are well defined and compact.

Given a map in $g\in B=\RHH(\W)\cap C^0(K_1,\C)$, consider 
$\left[L_1(g)\right](z)= \frac{1}{a} \psi'\circ\psi(a z) g(az)$. 
Since $\Cl\left(a\W\right)\subset K$ then the map 
$g(a\cdot)$ is in $\RHH(\W)\cap C^0(K,\C)$. 
Therefore, the map $L_1(g)$ belongs to $B$ which means that 
$L_1:B\rightarrow B$ is well defined. On the other hand we 
have that for any $g\in B$, $L_2(g)$ is defined as
$\left[L_2(g)\right](z) = \frac{1}{a} g\circ \psi(az)$. As 
the set $K$ has been considered such that 
$ \psi(a \W)\subset K$, $L_2$ is also well defined.

Consider $U$ the unit ball of $B$. 
Since $B$ is a
Banach space, to prove that $L_i$ is compact it is enough to prove that
 $L_i(U)$ is relatively compact (for $i=1,2$). 
This follows easily using proposition 9.13.1 of \cite{Die69}. 
For each compact set 
in $K'$ in $\W$ we have that the maps $L_i(U)$ are bounded,
then it follows that $L_i(U)$ is relatively compact in 
$C^0(K',\C)$. Concretely, we can take $K'=K$ and 
 we have that $L_i(U)\subset\RHH(\W)$ 
is relatively compact in $C^0(K,\C)$, therefore it is 
relatively compact in $B$. 
\end{proof}


\section{Reducibility loss and quasi-periodic renormalization}
\label{chapter application} 

In this section we use the renormalization operator 
to study the reducibility loss bifurcations of 
a two parametric family of q.p. forced map. 
Concretely, the main result of this section is a proof 
of the existence of reducibility loss bifurcations for a  two-parametric 
family of  q.p. forced map satisfying  
suitable conditions, see theorem \ref{thm existence of 
non-reducibility direction}. In particular this 
theorem applies to the case of the Forced Logistic 
Map considered in \cite{JRT11p}, but this will be 
discussed in \cite{JRT11c}.
The proof is not complete, in the sense 
that we need to assume the injectiveness of the renormalization operator 
$\TT_\omega$ (see conjecture {\bf \ref{conjecture H2}}). 

In section \ref{section boundaris of reducibility} we consider 
certain sets $\Upsilon^+_n(\omega)$ and $\Upsilon^-_n(\omega)$
of codimension one in the space $\BB$, which correspond to 
the reducibility loss of the attracting $2^n$-periodic orbit. 
We show that the intersection of these sets with the subset of uncoupled 
maps corresponds to the sets $\Sigma_n$ of maps (in the one dimensional 
case) such that its attracting $2^n$ periodic orbit is super-attracting. 
The main result of this section relates the sets $\Upsilon^+_n(\omega)$ 
(respectively $\Upsilon^-_n(\omega)$)
for different values of $n$ through the renormalization operator $\TT_\omega$. 

In section \ref{Section consequences for a two parametric family of maps} 
we consider a generic two parametric family of maps satisfying
certain hypotheses. We intersect the family with the previous 
$\Upsilon^+_n(\omega)$ and $\Upsilon^-_n(\omega)$. 
Differently to the one dimensional analogous, the intersection 
gives a one dimensional curve 
in the family, which corresponds to a 
reducibility loss bifurcation. Using the (quasi-periodic) 
renormalization operator and the results in section 
\ref{section boundaris of reducibility}  we prove 
the existence of reducibility loss bifurcation 
curves on the parameter space of the family considered. 
In theorem \ref{thm existence of 
non-reducibility direction} we show that, given 
a two parametric family which uncouples, there exist two
reducibility loss bifurcation curves  around the points 
such that the uncoupled map has a super-attracting periodic orbit. 
This fact was observed numerically in \cite{JRT11p}. 

We also discus the weakening of the hypothesis of 
theorem \ref{thm existence of non-reducibility direction} 
and we give explicit expressions of the slopes of the  reducibility loss bifurcations 
in terms of the family of maps and its iterates by the 
(quasi-periodic) renormalization operator.

As in previous sections the proofs have been moved to the 
end of each subsection. 
\subsection{Boundaries of reducibility}
\label{section boundaris of reducibility}

In this section we  work in the analytic framework, 
concretely we consider maps belonging to $\BB$ the space 
of the q.p forced one dimensional unimodal maps (as defined in 
subsection  \ref{subsection study of TT}). Again, 
let us consider the splitting  $\BB=\BB_0\oplus \BB_{0}^{c}$ 
given by the projection $p_0[f](x):= \int_{0}^{1} 
f(\theta,x)d\theta$, in other words the spaces given by  $\BB_0=p_0(\BB)$ and 
$\BB_0^c = (\mathrm{Id} - p_0)(\BB)$, where $\mathrm{Id}$ is the 
identity map. The renormalization operator for 
q.p. forced maps is  denoted by $\TT_{\omega}$, the 
renormalization operator for one dimensional maps by $\RR$ and 
the fixed point by $\Phi$ independently of the operator (recall that the 
fixed points of $\RR$ extend automatically to fixed points of $\TT_\omega$). Given a
map $F$ like (\ref{q.p. forced system interval}) with $f\in\BB$ and $\omega\in\T$ we 
denote by $f^{n}:\T \times \R \rightarrow \R$ the $x$-projection
of $F^n(x,\theta)$. Equivalently, $f^n$ can be defined through the recurrence
\begin{equation}
\label{definicio falan}
f^n(\theta, x) = f( \theta +(n-1) \omega, f^{n-1}(\theta,x)). 
\end{equation}

In this subsection, differently to the previous one, 
whenever  $\omega$ is used, it is assumed to 
be Diophantine. 
Let us denote by $\Omega$ 
the set of Diophantine numbers, 
this is $\Omega= \Omega_{\gamma,\tau}$ the set of $\omega\in \T$
such that there exists $\gamma >0$ and $\tau \geq 1$
such that
\[
|q w - p| \geq \frac{\gamma}{|q|^{\tau}}, \quad 
\text{ for all } (p,q) \in \Z \times (\Z \setminus \{0\}). 
\]

Additionally, we assume that the following conjecture is true.  

\begin{conj}
\label{conjecture H2} 
The operator $\TT_{\omega}$ (for any $\omega\in\Omega$) is an injective 
function when restricted to the domain $\BB\cap \DD(\TT)$. Moreover, 
there exist $U$ an open set of $\DD(\TT)$ containing $W^u(\Phi,\RR)
\cup W^s(\Phi,\RR)$\footnote{Here $W^s(\Phi,\RR)$ and $W^u(\Phi,\RR)$ 
are considered as the inclusion in $\BB$ of the 
stable and the unstable manifolds of the fixed point $\Phi$ (given by 
{\bf H0}) by the map $\RR$ in the topology of $\BB_0$ (the inclusion 
of one parametric maps in $\BB$).}
where the operator $\TT_\omega$ is differentiable.
\end{conj}

The first part of the conjecture is proved for the 
one dimensional case in \cite{MvS93}. The proof consists 
of, given two maps with the same image by the operator, first to show 
that their renormalization interval is the same and then to expand 
the image of the maps around their fixed point and then deduce 
that the original maps are the same maps. With our approach to the 
quasi-periodic case, as we do not have an equivalent concept to 
renormalization interval, the 
same argument is no longer applicable. Despite the analogy 
with the one dimensional case, there is, a priori, no 
evidence about the conjecture.  A posteriori, we 
have that the results obtained assuming that the conjecture is true
are coherent with the dynamics of the Forced Logistic Map. 
In \cite{JRT11c} we compute numerically the slopes 
of the reducibility loss bifurcations of the FLM by 
two independent methods. The first method is computing the slope using 
the dynamical characterization of the bifurcations. The 
second one is using the formulas given in corollary 
\ref{corollary non-reducibility directions}. 
Both coincide up to a reasonable accuracy.

The second part of the conjecture is only introduced to simplify
the forthcoming discussion, but it can be avoided if necessary. 
See remark \ref{remark avoiding second part of conjecture H2)} 
for details. 

Whenever the conjecture {\bf \ref{conjecture H2}} is needed for a 
result it is explicitly stated in the hypotheses. 

Let $K_0>0$ be a fixed constant value. Then we can consider the sets
\[ 
\Upsilon^+_{n} (\omega)= \left\{ f \in \BB \left| \begin{array}{c} 
\displaystyle \text{ There exists } x_0 \in \RHH(\B_\rho,\W) 
\text{ with } x_0(\theta + 2^n \omega) = f^{2^n} (\theta,x_0(\theta))
\text{ s. t. } \\ 
\rule{0pt}{3ex}
\displaystyle 
\int_{0}^{1} \ln|  D_x f^{2^n}(\theta,x_0(\theta) )| d\theta < - K_0
\text{ and }\min_{\theta\in \T} D_x \left( f^{2^n}\right)(\theta, x_0(\theta))=0.  
\end{array} \right. \right\},
\]
and
\[ 
 \Upsilon^-_{n}(\omega) = \left\{ f \in \BB \left| \begin{array}{c} 
\displaystyle \text{ there exists } x_0 \in  \RHH(\B_\rho,\W)
\text{ with } x_0(\theta + 2^n \omega) = f^{2^n} (\theta,x_0(\theta))
\text{ s. t. } \\ 
\rule{0pt}{3ex}
\displaystyle 
 \int_{0}^{1} \ln|  D_x f^{2^n}(\theta,x_0(\theta) )| d\theta < -K_0 
\text{ and }\max_{\theta\in \T} D_x \left( f^{2^n}\right) (\theta, x_0(\theta))=0.
\end{array} \right. \right\}.
\]

Note that (due to the two first conditions) these sets are 
contained in the set of all the maps in $\BB$ which have 
a $2^n$-periodic attracting curve. We require 
the integral being less than $-K_0$ instead of being less than $0$ 
for technical reasons. The third condition is imposed with the 
aim that these sets correspond to the bifurcation manifold 
associated to the reducibility loss. We have the following properties 
which give a good characterization of these sets.

\begin{prop} 
\label{prop relacio lambda upsilon}
Let $\Sigma_{n}$ be the inclusion in $\BB$ of the 
set of one dimensional unimodal maps with 
a super-attracting $2^n$ periodic orbit.
We have  that
\[\Upsilon^+_{n}(\omega)\cap \BB_0 = \Upsilon^-_{n}(\omega) \cap \BB_0 = \Sigma_{n}, \]
for any $\omega\in \Omega$. 
\end{prop}

\begin{prop}
\label{prop upsilon banach manifold}
Let $f\in\Upsilon^+_n(\omega)$ (respectively $f\in\Upsilon^-_n(\omega)$) and  let $x$ be its $2^n$-periodic curve. If
$D_x \left(f^{2^n}\right)(\theta, x(\theta))$ has a unique non-degenerate absolute
minimum\footnote{It can have several local minima but the absolute minimum has to be 
unique and not degenerate.} (respectively maximum), then 
$\Upsilon^+_n(\omega)$  (respectively  $f\in\Upsilon^-_n(\omega)$) 
is a codimension one 
manifold in a neighborhood of $f$.
\end{prop}

\begin{prop}
\label{prop upsilon is reducibility loss} 
Let $\{f_\mu\}_{\mu\in A}$ be a one parametric family of maps 
such that: 
\begin{enumerate}
\item There exist a parameter value $\mu_0$ for which the family crosses 
$\Upsilon_n^+(\omega)$ (respectively $\Upsilon_n^-(\omega)$) transversely at $\mu=\mu_0$. 
\item 
Consider $x_{\mu_0}$ the $2^n$ periodic 
invariant curve of $f_{\mu_0}$ given by the definition of 
$\Upsilon_n^+(\omega)$ (respectively  $\Upsilon_n^-(\omega)$) 
such that 
$D_x\left(f^{2^n}_{\mu_0}\right)(\theta,  x_{\mu_0}(\theta))$ has a 
unique non-degenerate minimum (respectively maximum). 
\end{enumerate}
Then we have that the invariant periodic curve $x_{\mu_0}$ extends to a
periodic invariant curve\footnote{To extend the periodic invariant curve 
we need to reduce $\rho$ (the width of the band of analyticity 
with respect to $\theta$) to be small enough in terms of $K_0$, but this 
reduction of $\rho$ is 
done only once.} $x_{\mu}$ of $f_\mu$ for any $\mu$ in 
an open neighborhood of $\mu_0$. Additionally, this invariant curve 
undergoes a reducibility loss bifurcation at $\mu=\mu_0$. 
\end{prop}

Let us introduce some notation to state the next result. Consider the map, 
\begin{equation}
\label{equation map T}
\begin{array}{rccc}
T:& \T \times \DD(\TT) &\rightarrow &  \T \times  \BB  \\
\displaystyle \rule{0pt}{3ex} & (\omega,f) & \mapsto &  \rule{0pt}{3ex}
\displaystyle  (2\omega, \TT_\omega (f)), 
\end{array} 
\end{equation}
where $\TT_\omega$ is the renormalization operator for q.p. forced maps, 
as in section \ref{section definition of q.p. renormalization}, 
and the set $\DD(\TT) \subset \BB$ is its domain of definition. 
Recall that to have $\TT_\omega$ well defined it is not necessary 
$\omega\in\Omega$ (i.e. $\omega$ Diophantine). 
Additionally, if $\omega\in \Omega$ then 
we have that $2^k\omega \in \Omega$ for any $k\in \Z$, therefore 
$T(\Omega \times \DD(\TT)) \subset \Omega \times \BB$.
On the other hand we have that the 
sets $\Upsilon_n^+(\omega)$ and $\Upsilon_n^-(\omega)$ 
are only defined for $\omega \in \Omega$.

\begin{defin}
\label{definition times renormalizable}
We will say that a pair $(\omega, f) \in \T \times \BB$ is {\bf $n$-times 
renormalizable} if $T^k(\omega,f) \in \T \times \DD(\TT) $ for 
$k=0,\dots, n-1$. 
\end{defin}

Consider a pair $(\omega,f_0)\in \Omega\times\BB$  with a $2$-periodic 
invariant attracting curve $x_0$ with rotation number $\omega$.
 Assume that the Lyapunov exponent of the curve is less that $-K_0$, 
with $K_0>0$ a fixed value, in other words,
\[
 \int_{0}^{1} \ln|  D_x f_0^{2}(\theta,x_0(\theta) )| d\theta < -K_0 . \]
In forthcoming lemma \ref{Lemma proof prop upsilon banach} we will prove that
the persistence of 
the invariant curve extends to a neighborhood of $f_0$ 
(if the width $\rho$ of the band $\B_\rho$ around the torus 
$\T$ is small enough with respect to $K_0$). 
Let $V\subset \BB$ be this neighborhood, and let $x: 
V \rightarrow \RHH(\B_\rho,\C)$, be the period $2$ invariant curve 
associated to $f$. Here $\RHH(\B_\rho,\C)$ denotes the space of functions 
$f: \B_\rho\rightarrow \C$ which are real analytic in $\B_\rho$ and 
continuous in its closure. Then we can define 
the map $G_1 $ as
\begin{equation}
\label{equation definition G1}
\begin{array}{rccc}
G_1:& \Omega\times V &\rightarrow & \RHH(\B_\rho,\C) \\
\rule{0pt}{3ex} &  (\omega,g) & \mapsto &  
D_x g \big(\theta+\omega,g(\theta, \left[x(\omega,g)\right](\theta))\big) 
D_x g \big(\theta,\left[x(\omega,g)\right](\theta)\big). 
\end{array}
\end{equation}
On the other hand, consider the minimum and the 
maximum as operator between spaces of functions:
\begin{equation}
\label{equation definition of minim} 
\begin{array}{rccc}
m:& \RHH(\B_\rho,\C) &\rightarrow & \R \\
\rule{0pt}{3ex} & g & \mapsto & \displaystyle \min_{\theta \in \T} g(\theta). 
\end{array}
\end{equation}
and
\begin{equation}
\label{equation definition of maximum} 
\begin{array}{rccc}
M:& \RHH(\B_\rho,\C) &\rightarrow & \R \\
\rule{0pt}{3ex} & g & \mapsto & \displaystyle \max_{\theta \in \T} g(\theta). 
\end{array}
\end{equation}

We have the following theorem, which relates the manifolds 
$\Upsilon_n^+(\omega)$ and $\Upsilon_n^-(\omega)$ for different $n$ through the 
renormalization operator. 

\begin{thm}
\label{theorem local chard of upsilon n} 
Let $\omega\in\Omega$ and $f\in \Upsilon_n^+(\omega)$, respectively 
$f\in \Upsilon_n^-(\omega)$,  be 
a function such that the pair $(\omega,f)$ is $n-1$ times renormalizable. 
Additionally assume that conjecture {\bf \ref{conjecture H2}} is true. 
Then there exist $U$ a neighborhood (in $\BB$)  of $f$ such that 
\[
U\cap \Upsilon_n^+(\omega) = \{ f\in U | \thinspace G_1^+(T^{n-1} (\omega,f) ) =0\},
\]
respectively 
\[
U\cap \Upsilon_n^-(\omega) = \{ f\in U | \thinspace G_1^-(T^{n-1} (\omega,f) ) =0\},
\]
where
\[
G^+_1(\omega, g) := m\circ G_1(\omega,g),
\]
and respectively 
\[
G^-_1(\omega, g) := M\circ G_1(\omega,g),
\]
with $G_1$, $m$ and $M$ defined in equations (\ref{equation definition G1}),
(\ref{equation definition of minim})  and (\ref{equation definition of maximum}). 
\end{thm}


{\bf Proofs} 

\begin{proof}[Proof of proposition \ref{prop relacio lambda upsilon}]
We will do the proof only for the case $\Upsilon^+_n(\omega)$. The case 
$\Upsilon^-_n(\omega)$ is completely analogous. 

If we have a map $f_0\in \Sigma_n\subset \BB_0$, 
it has a super-attracting periodic orbit $x_0$, then we 
have that its Lyapunov exponent is $-\infty$, since
$D_x f^{2^n} (\theta,x_0) \equiv 0$. Therefore $\Sigma_n \subset \Upsilon_n^+(\omega) \cap \BB_0 $.

Consider $f\in \Upsilon_{n}^+\cap \BB_0$. Since 
$f$ is in $\BB_0$ it does not depend on $\theta$. Consider 
also $h=f^{2^n}$, 
which neither depends on $\theta$.  On the other hand 
using that $f$ is in $\Upsilon_{n}^+(\omega)$ we have that there exists a function 
$x\in \RHH(\B_\rho,\W)$ satisfying the following invariance equation
\[x(\theta +2^n \omega) = h(x(\theta)). \]
Differentiating the invariance equation we have 
\[ x'(\theta+ 2^n \omega) =   h'(x(\theta)) x'(\theta). \]
From the fact that $f \in \Upsilon^+_{n}(\omega)$ it follows that there exist a $\theta_0$ such 
that $h'(x(\theta_0))= 0$. Using the last equation we have that $x'$ is 
zero in a dense subset of $\T$ and using the continuity of 
$x'$ we have that $x'\equiv 0$, therefore $x$ is constant. 
Finally note that as $x=h(x)=f^{2n}(x)$ and 
$D_x f^{2^n}(x)=0$, we have that  $f$ belongs to $\Sigma_n$. 
\end{proof}


\begin{proof}[Proof of proposition \ref{prop upsilon banach manifold}]
As before we prove only the case of $\Upsilon_n^+$. We start with 
a preliminary lemma. 
\begin{lem}
\label{Lemma proof prop upsilon banach}
Consider $g_0\in \BB=\BB(\B_\rho,\W)$ and $x_0$ an invariant curve of $g_0$. 
Assume that 
\begin{equation}
\label{eqn lemma prop upsilon banach}
\int_{0}^{1} \ln|  D_x g_0(\theta,x_0(\theta) )| d\theta < - K < 0.
\end{equation}
Then, for a sufficiently small value of $\rho$, there exist a neighborhood 
$U$ of $g_0$ and a smooth function $x:U\rightarrow \BB$ such that 
$x(g)$ is an invariant curve of $g$ for any $g\in U$ and 
$x(g_0)=x_0$. 
\end{lem} 
\begin{proof}
This lemma corresponds to the analytic version of the continuation 
problem of an invariant curve. In \cite{JT08} it is studied 
the $C^r$ version of this problem. 
The authors prove that the curve can be continued if $1$ does not belong 
to the spectrum of the transfer operator $\LL$ associated
 to the problem (see section 3.3 of \cite{JT08}). Then it is
shown that the spectrum of $\LL$ is contained in the disk 
of radius $b=\exp\left(\int_{0}^{1} \ln|  
D_x g_0(\theta,x_0(\theta) )|d\theta\right)$.

Let $\B_\rho$ be a band of width $\rho$ around the real torus $\T$, 
and $\HH(\B_\rho,\C)$ denote the space of holomorphic maps from 
$\B_\rho$ to $\C$ and continuous on the closure of $\B$. 
Note that transfer operator $\LL$ can be considered both, in
$C^r(\T,\R)$  endowed  with the standard $C^r$-norm,
or in $\HH(\B_\rho,\C)$ endowed with the supremum norm. To
distinguish the spectrum of the transfer operator with
respect to the norm considered we will denote
each of the respective cases by $ \Spec(\LL, C^r)$ or
$\Spec(\LL,\HH)$. Using theorem 9.2 of \cite{HL07},
we have that
\[
\partial \Spec(\LL,\HH) \subset  \Spec(\LL,C^r) + O(\rho), 
\]
for $\rho>0$ small enough.

The notation $A\subset B + O(\rho)$ means that
there exists a constant $C>0$ such that for any $a\in A$
there exists $b\in B$ with $d(a,b)\leq C\rho$.

Using equation (\ref{eqn lemma prop upsilon banach}) we have 
that $b<1$. Then there exists a sufficiently small  $\rho$ 
such that $b+\rho<1$. Then using that $\partial \Spec(\LL,\HH)$ 
is contained in the disc of radius $b+\rho<1$  one can 
extend the persistence of invariant attracting curves to the analytic case.
\end{proof}

Consider an arbitrary function $f_0\in \Upsilon_n^+(\omega)$. We have that 
there exists  $x_0: \B_\rho\rightarrow \W$ which is a $2^n$-invariant
 attracting curve of $f_0$. We can consider the auxiliary function 
\[
\begin{array}{rccc}
F:& \BB\times \RHH(\B_\rho,\W) &\rightarrow & \RHH(\B_\rho,\W)  \\
\rule{0pt}{3ex} & (g,x )  & \mapsto & [F(g,x)](\theta):= x(\theta + 2^{n} \omega) - 
g^{2^n}(\theta,x(\theta)).
\end{array}
\]

We have that the Lyapunov exponent of the curve is less than $-K_0$.
Using the lemma \ref{Lemma proof prop upsilon banach} we
have that, if the width $\rho$ of the band $\B_\rho$ around the torus
$\T$ is small enough with respect to $K_0$,
then  there exists a neighborhood $U_n$ of 
$f_0$ in $\BB$ and a function 
\[
\begin{array}{rccc}
x:& U_n &\rightarrow & \RHH(\B_\rho,\W) \\
\rule{0pt}{3ex} &  g & \mapsto & x(g),
\end{array}
\]
with $x(f_0)=x_0$ and such that $x(f)$ is a $2^n$ periodic curve of 
$f$ for any $f\in U_n$. Moreover, due to the continuity 
of the Lyapunov exponent, we can suppose that the Lyapunov 
exponent of $x(g)$ is negative for any $g\in U_n$, replacing 
$U_n$ by a smaller neighborhood if necessary.

Now we consider the auxiliary function 
\begin{equation}
\label{equation definition tildeGn}
\begin{array}{rccc}
\tilde{G}_n:& U_n &\rightarrow & \RHH(\B_\rho,\C) \\
\rule{0pt}{3ex} &  f & \mapsto & [\tilde{G}_n(f)](\theta):= 
D_x \left(f^{2^n}\right) \left(\theta, [x(f)](\theta)\right),
\end{array}
\end{equation}
and consider also the minimum operator defined as in
(\ref{equation definition of minim}). 

By hypothesis we have that $ \left[D_x \left(f_0^{2^n}\right)\right](\theta, [x(f_0)](\theta))$ has 
a unique minimum, therefore we can apply proposition
\ref{proposition diferentiability of minimum} (in appendix \ref{appendix 
minimum function}). The uniqueness of the minimum extends
to $U_n$, replacing it again by a smaller neighborhood if necessary. Therefore 
we have that the map $G_n:U_n \rightarrow \R$ defined as $m\circ \tilde{G}_n$ 
is differentiable. Let us remark that $G_n$ depends indeed on $\omega$, but the 
differentiability is only needed with respect to $g$. 
Then we have that 
\begin{equation}
\label{equation local chard of Upsilon} 
\Upsilon_n^{+}(\omega) \cap U_0= \{ f \in U |\thinspace G_n (\omega,f)=0\},
\end{equation}
which completes the proof. 
\end{proof}


\begin{proof}[Proof of proposition \ref{prop upsilon is reducibility loss}]
As before we only consider the case involving $\Upsilon_n^{+}(\omega)$. 
Consider $f_{\mu_0}$ the intersection of the family with the set  
$\Upsilon_n^{+}(\omega)$. 
Using the second hypothesis of the proposition we are under the 
same hypothesis of proposition 
\ref{prop upsilon banach manifold}. Following the proof of this proposition 
we have that there exists $U_0$ an open neighborhood of $f_{\mu_0}$ such that 
$\Upsilon_n^{+}(\omega) \cap U_0$ is given by equation 
(\ref{equation local chard of Upsilon}) with $G_n= m\circ \tilde{G}_n$ and 
the maps $m$ and 
$\tilde{G}_n$ are given by equations (\ref{equation definition of minim})
and (\ref{equation definition tildeGn}). Moreover we also have 
that  $G_n$  is differentiable in $U_0$.  

Using that the family $f_\mu$ crosses transversely the manifold 
$\Upsilon_n^{+}(\omega) $
we have that\\  $\partial_\mu G_n(f_\mu)_{|\mu=\mu_0} \neq 0$. Actually we can assume 
that $\partial_\mu G_n(f_\mu)_{|\mu=\mu_0} < 0$, otherwise we can 
replace $\mu$ by $\tilde{\mu}= 2\mu_0 - \mu$ and consider the family 
$f_{\tilde{\mu}}$ instead of $f_\mu$.  Recall also that 
\[
G_n(f_{\mu})  = \min_{\theta\in \T} D_x \left( f^{2^n}_{\mu}\right)(\theta, x_{\mu}(\theta)).
\]
Then for  $\mu < \mu_0$  we have $G_n(f_\mu)>0$ and therefore we have 
that $D_x \left( f^{2^n}_{\mu_0}\right)(\theta, x_{\mu_0}(\theta))>0$ 
for any $\theta\in \T$. Using corollary 1 of \cite{JT08}
we have that $x_{\mu}$ is reducible. 
Moreover, due to second hypothesis of the proposition we have that 
$D_x \left( f^{2^n}_{\mu_0}\right)(\theta, x_{\mu_0}(\theta))$  
has a double zero $\theta_0$. Finally using the transversality hypothesis
we have that $\partial_\mu \left(D_x \left( f^{2^n}_{\mu_0}\right)
(\theta_0, x_{\mu_0}(\theta_0))\right) \neq 0$. This proves that $x_{\mu}$ undergoes
 a reducibility loss bifurcation. 
\end{proof}


\begin{proof}[Proof of theorem \ref{theorem local chard of upsilon n}]
Once again, we only consider the case involving $\Upsilon_n^{+}(\omega)$, since 
the other case is completely analogous. Let us introduce the following lemma 
for the proof of theorem.

\begin{lem}
\label{lemma invariancia upsilons}
Let $\mathrm{Im} (\TT_\omega)$ be the image of the operator $\TT_\omega$.
Assume that we have $\omega\in \Omega$, then we 
have that
\begin{equation}
\TT_\omega(\Upsilon_{n}^+ (\omega) \cap \DD(\TT)) = 
 \Upsilon_{n-1}^{+}(2\omega) \cap \mathrm{Im} (\TT_\omega) .
\end{equation}
\end{lem}
\begin{proof}
Given $f\in \TT_\omega(\Upsilon^+_{n}(\omega) \cap \DD(\TT))$, we 
have that there exists a function $g \in \Upsilon^+_n(\omega)$ such that
\begin{equation}
\label{equation demo prop invariancia lambdas}
f(\theta, x) =\frac{1}{a}g(\theta+\omega,g(\theta,ax)), 
\end{equation}
with $a= \int_0^{1}g(\theta, 1) d\theta$.
Note that the rotation number when we compose $g$ with itself
is $\omega$, while the rotation number when we compose 
$f$ with itself is $2\omega$; this is not obvious from the notation 
but it is important to have it in mind for this proof.

To prove the first inclusion it is enough to check that $f$ is 
in $\Upsilon^+_{n-1}(2\omega)$ since $f$ is trivially in $\mathrm{Im}(\TT_\omega)$. 
 
Using that $g$ is in $\Upsilon^+_{n}(\omega)$ we have that there exists 
$x_0: \B_\rho \rightarrow \W$ with  
\[x_0(\theta + 2^n \omega) = g^{2^n} (\theta,x_0(\theta)), \]
and such that 
\[\min_{\theta\in \T} D_x \left( g^{2^n}\right)(\theta, x_0(\theta))=0.\]

Using equation (\ref{equation demo prop invariancia lambdas}) it 
is easy to check that 
\begin{equation}
\label{equation demo prop invariancia lambdas 2}
f^{2^{n-1}} (\theta,x) = \frac{1}{a} g^{2^n}(\theta, ax),
\end{equation} 
and 
\begin{equation}
\label{equation demo prop invariancia lambdas 3}
D_x \left(f^{2^{n-1}} \right)  (\theta,x) = D_x\left(g^{2^n}\right) 
(\theta, ax),
\end{equation}
for any $\theta\in\B_\rho$ and $x\in \W$. 

Consider $x_1$ the function defined as 
$x_1(\theta)=\frac{1}{a} x_0(\theta)$. From the last two equalities it follows that 
$x_1(\theta + 2^{n-1} (2 \omega)) = f^{2^{n-1}} (\theta,x_1(\theta))$, and 
\[\min_{\theta\in \T}  D_x \left( f^{2^{n-1}}\right)(\theta, x_1(\theta)) = 
\min_{\theta\in \T}  D_x \left( g^{2^{n}}\right)(\theta, x_0(\theta)) = 0.
\]
Therefore $f$ is in $\Upsilon^+_{n-1}(2\omega)$.

Let us see the converse inclusion. Consider 
$f\in  \Upsilon^+_{n-1}(2\omega) \cap \mathrm{Im} (\TT_\omega)$. 
Since $f$ is in $\mathrm{Im} (\TT_\omega)$  we have that 
there exists $g\in \DD(\TT)$  with $f=\TT_\omega(g)$. Therefore 
we only have to prove that $g$ is in $\Upsilon^+_n(\omega)$.
Using $f=\TT_\omega(g)$ one has that equation
(\ref{equation demo prop invariancia lambdas}) is satisfied again and 
this implies that 
equations  (\ref{equation demo prop invariancia lambdas 2})  and 
 (\ref{equation demo prop invariancia lambdas 3}) also hold. 
From $f \in \Upsilon^+_{n-1}(2\omega) $ we have that 
there exists a function  $x_1:\B_\rho \rightarrow \W$ with 
$x_1(\theta + 2^{n-1} (2 \omega)) = f^{2^{n-1}} (\theta,x_1(\theta))$ and 
\[
\min_{\theta\in \T}  D_x \left( f^{2^{n-1}}\right)(\theta, x_1(\theta))=0.
\] 
Consider now $x_0(\theta) := a x_1(\theta)$, then using equation  
(\ref{equation demo prop invariancia lambdas 2})  we have that 
\[ a x_1(\theta + 2^{n-1} (2 \omega)) =  g^{2^n}(\theta, ax_1(\theta)),\]
for any $\theta \in \T$. Using 
(\ref{equation demo prop invariancia lambdas 3})  we obtain 
\[
\min_{\theta\in \T}  D_x \left( g^{2^{n}}\right)(\theta, x_1(\theta)) = 0, 
\] 
what completes the proof of the lemma.  
\end{proof}


The proof will follow by induction. Note that the case $n=1$ is satisfied trivially. 
Then we can assume that the case $n-1$ is true and check the case $n$. 

We have that $(f,\omega)$ is renormalizable, 
then $f\in D(\TT)\cap \Upsilon^+_n(\omega)$. Using conjecture 
{\bf \ref{conjecture H2}} we have that a point belongs to $\Upsilon^+_n(\omega)\cap \DD(\TT)$ if, 
and only if, $\TT_\omega(g)$ belongs to 
$\TT_\omega(\Upsilon^+_n(\omega)\cap \DD(\TT))$. Using lemma 
\ref{lemma invariancia upsilons} we have that 
$\TT_\omega(\Upsilon^+_n(\omega)\cap \DD(\TT)) 
=  \Upsilon^+_{n-1}(2\omega) \cap \mathrm{Im} (\TT_\omega)$. 

At this point we need to consider the case $n=2$ independently. 
In the case $n=2$ we have that $\TT_\omega(\Upsilon^+_2(\omega)\cap \DD(\TT)) 
=  \Upsilon^+_{1}(2\omega) \cap \mathrm{Im} (\TT_\omega)$, then we have 
that there exists $U_1$ a neighborhood of $\TT_\omega(f)$ such 
that 
\[
U_{1} \cap \Upsilon^+_{1}(2 \omega) = \{g \in U | \thinspace G_1( 2 \omega,g) =0 \}.
\] 
Consider $U_2= \TT_\omega^{-1}(U_1)$, using that $\TT_{\omega}$ is 
continuous, we have that $U_2$ is an open neighborhood of $f$. Then we have
\[U_2 \cap \Upsilon^+_2(\omega) 
= \TT_{\omega} ^{-1} (U_{1}\cap \Upsilon^+_{1}(2\omega) ) =
\{ f \in U_{1} | \thinspace G_1( 2 \omega,\TT_\omega(f) ) = G_1(T(\omega,f)) =0 \},
\] which finishes the proof for the case $n=2$. 

In the case $n>2$ we have that 
the pair $(2\omega, \TT_\omega(f))$  is $n-2$ times renormalizable
We apply now the induction hypothesis, then we have that 
there exists $U_{n-1}$ a neighborhood of $\TT_\omega(f)$ such that 
\[
U_{n-1} \cap \Upsilon^+_{n-1}(2 \omega) = \{g \in U | \thinspace G_1(T^{n-2} (2 \omega,g))  =0 \}.
\]

Consider $U_n=\TT_{\omega}^{-1}(U_{n-1})$, since $\TT_{\omega}$ is a 
continuous function, we have that $U_n$ is an open neighborhood of $f$ and 
then we have
\[U_n \cap \Upsilon^+_n(\omega) 
= \TT_{\omega} ^{-1} (U_{n-1}\cap \Upsilon^+_{n-1}(2\omega) ) =
\{ f \in U_{n-1} | \thinspace G_1(T^{n-2} (2 \omega,\TT_\omega(f)) ) =0 \}.\]
Using that $T^{n-2} (2 \omega,\TT_\omega(f)) = T^{n-1} ( \omega,f)$
the proof is finished. 
\end{proof}


\subsection{Consequences for a two parametric family of maps} 
\label{Section consequences for a two parametric family of maps}


Consider a two parametric family of maps like 
(\ref{q.p. forced system interval}). For the rest of this 
section we assume that $\omega$ is a
fixed Diophantine number ($\omega\in \Omega$). Then the family of
maps is determined by two parametric family of maps
 $\{c(\alpha,\eps)\}_{(\alpha,\eps)\in A}$
contained in $\BB$ (concretely they are unimodal q.p. forced maps),
where $A = [a,b]\times[0,c]$ and $a$, $b$ and $c$ are real numbers
(with $a<b$ and $0<c$).

We assume that the dependency on the parameters is analytic, then
the family can be thought as an analytic map
$c:A \rightarrow \BB$. In this subsection we prove (under
suitable hypotheses) the existence of reducibility loss bifurcations like the
ones observed in the numerical computations of the Forced Logistic
Map \cite{JRT11p, JRT11c}.

We will consider families of maps satisfying the following hypothesis.
\begin{description}
\item[H1)]  The family $\{c(\alpha,\eps)\}_{(\alpha,\eps)\in A}$ uncouples for
$\eps=0$, in the sense that the family $\{c(\alpha,0)\}_{\alpha\in[a,b]}$                
does not depend on $\theta$ and it
is a one parametric family of unimodal maps has a full cascade of                
period doubling bifurcations. We assume that the family               
$\{c(\alpha,0)\}_{\alpha\in[a,b]}$ crosses transversally the 
stable manifold of $\Phi$ the fixed point
of the renormalization operator and each of
the manifolds $\Sigma_n$ for any $n\geq 1$, where
$\Sigma_{n}$ is the inclusion in $\BB$ of the
set of one dimensional unimodal maps with
a super-attracting $2^n$ periodic orbit.
\end{description}

In other words, we assume that the family $c(\alpha,\eps)$ 
can be written as, 
\[
c(\alpha,\eps)=c_0(\alpha) + \eps c_1(\alpha,\eps),
\]
with $\{c_0(\alpha)\}_{\alpha\in[a,b]}\subset \BB_0$  having a full 
cascade of period doubling bifurcations. 

Given a family $\{c(\alpha,\eps)\}_{(\alpha,\eps)\in A}$ satisfying 
hypothesis {\bf H1} let $\alpha_n$ be the parameter value for 
which the uncoupled family $\{c(\alpha,0)\}_{\alpha\in[a,b]}$
intersects the manifold $\Sigma_n$. Note that the critical 
point of the map $c(\alpha_n,0)$ is a $2^n$-periodic orbit.
The main goal of this subsection is to prove that for every 
parameter value $(\alpha_n,0)$ there are two curves in 
the parameter space, one corresponding to a reducibility 
loss bifurcation and the other one corresponding to a reducibility 
recover. These curves are born at the point $(\alpha_n,0)$ of 
the parameter space. 

Consider a map $f_0\in \BB$ and $\omega \in \Omega$, 
such that $f$ has a 
periodic invariant curve $x_0$ of rotation number $\omega$
with a  Lyapunov exponent less or equal than certain 
$-K_0<0$. Recalling the arguments in 
the proof of proposition \ref{prop upsilon banach manifold} we have that 
there exist a neighborhood $V\subset \BB$ of $f_0$ and a map 
$x\in \RHH(\B_\rho, \W)$ such that $x(f)$ is a periodic 
invariant curve of $f$  for any $f\in U_0$. Concretely, 
if we have a map $f_0\in \BB$  with a $2$-periodic 
invariant attracting curve, we can define the map $G_1:\Omega\times V \rightarrow
\RHH(\B_\rho,\C)$ as in (\ref{equation definition G1}).

On the other hand, we can consider the counterpart of the map 
$G_1$ in the uncoupled case. Given a map $f_0\in \BB_0$, consider 
$U\subset \BB_0$ a neighborhood of $f_0$ in the $\BB_0$ topology. 
Assume that $f_0$ has a attracting $2$-periodic 
orbit $x_0\in I$. We have  that $x$ depends analytically on the map,
 therefore it induces a map $x:U \rightarrow \W$. Then if we take $U$ small 
enough we can suppose that there exists an analytic map $x:U\rightarrow \W$
such that $x[f]$ is a periodic orbit of period 2. 

Now consider the map
\begin{equation}
\label{equation definition widehatG1}
\begin{array}{rccc}
\widehat{G}_1:& U\subset \BB_0 &\rightarrow & \C \\
\rule{0pt}{3ex} &  f & \mapsto & D_x f \big(f(x[f])\big) D_x f \big( x[f] \big). 
\end{array}
\end{equation}
Let us remark that the zeros of this map define locally the manifold $\Sigma_1$. 
On the other hand it corresponds to the map $G_1$ restricted to 
the space $\BB_0$, despite the fact that $\widehat{G}_1(f)$ has to be 
seen as an element of $\RHH(\B_\rho,\W)$.

At this point we need to introduce an additional hypothesis on the family 
$\{c(\alpha,\eps)\}_{(\alpha,\eps)\in A}$. Consider $w_k=2^k w_0$ for any 
$k\geq 0$ and $f^{(n)}_k=\RR^k (c(\alpha_n,0))$. We have that  $f^{(n)}_0$ 
tends to $W^s(\RR,\Phi)$ when $n$ grow. Then  $\{f^{(n)}_k\}_{0\leq k <n}$ 
attains to $W^s(\RR,\Phi) \cup W^u(\RR,\Phi)$ and consequently there 
exist $n_0$ s. t. $\{f^{(n)}_k\}_{0\leq k <n} \subset U$ , 
where $U$ is the neighborhood given in conjecture {\bf \ref{conjecture H2}}. 
Therefore the operator $\TT_\omega$ is differentiable in the 
orbit $\{f^{(n)}_k\}_{0\leq k <n} \subset U$.  Consider the following 
hypothesis. 

\begin{description}
\item[H2)]  The family $\{c(\alpha,\eps)\}_{(\alpha,\eps)\in A}$ is such that 
\[
D G_1 \left(\omega_{n-1}, f^{(n)}_{n-1}\right) 
D\TT_{\omega_{n-2}}\left(f^{(n)}_{n-2}\right) \cdots 
D\TT_{\omega_0}\left(f^{(n)}_0\right) \partial_\eps c(\alpha_n,0),
\] 
has a unique non-degenerate minimum (respectively maximum) as a function from $\T$ to $\R$, 
for any $n\geq n_0$. 
\end{description}

Note that $c(\alpha_n,0)\in \Sigma_n$,  
therefore $f^{(n)}_{n-1} \in \Sigma_1$, 
consequently the function $G_1$ is defined at the point $f^{(n)}_{n-1}$. 
Hypothesis {\bf H2} is rather technical and 
not very intuitive. Further on in this section we show
that it is actually satisfied by maps like the Forced Logistic Map.

We have the following result, which ensures the existence of 
reducibility-loss bifurcations curves in the $(\alpha,\eps)$-plane
of parameters near the points $(\alpha_n, 0)$. This is one of the main 
results of this chapter. On the one hand it gives the existence of 
reducibility-loss bifurcations, but on the other hand it also gives 
explicit expression of these bifurcations in term of 
the renormalization operator $\TT_\omega$. 

\begin{thm}
\label{thm existence of non-reducibility direction}
Consider a family of maps $\{c(\alpha,\eps)\}_{(\alpha,\eps)\in A}$ as before
 such that hypotheses  {\bf H1} and {\bf H2} are satisfied
 and consider $\alpha_n$ the parameter values where the uncoupled 
family intersects the manifolds $\Sigma_n$ as above. 
Suppose 
that the rotation number of the system is equal to $\omega_0\in \Omega$. 
Assume also that conjecture {\bf \ref{conjecture H2}} is true.  
Then there exists $n_0$ such that, for any $n\geq n_0$, 
two curves are  born from every parameter  value $(\alpha_n, 0)$, 
locally expressed as $(\alpha_n^+(\eps),\eps)$ and 
$(\alpha_n^-(\eps), \eps)$, such that they correspond to a 
reducibility-loss bifurcation of the $2^n$-periodic invariant curve. 

Consider also the sequences
\begin{equation}
\label{equation sequences corollary directions}
\begin{array}{rcll} 
\omega_k  & = & 2 \omega_{k-1}  & \text{ for }  k=1,..., n-1. \\ 
\rule{0ex}{4ex} 
f^{(n)}_k & = & \RR\left(f_{k-1}^{(n)}\right)  
& \text{ for }  k=1,..., n-1. \\ 
\rule{0ex}{4ex} 
u^{(n)}_k & = & D \RR \left(f^{(n)}_{k-1}\right) u^{(n)}_{k-1} 
&  \text{ for } k = 1, ..., n-1.\\
\rule{0ex}{4ex} 
v^{(n)}_k & = & D \TT_{\omega_{k-1}}  \left(f^{(n)}_{k-1}\right) v^{(n)}_{k-1} 
&  \text{ for } k = 1, ..., n-1.
\end{array}
\end{equation}
with
\begin{equation}  
\label{equation sequences directions initial}
f^{(n)}_0  =  c(\alpha_n,0), \quad 
u^{(n)}_0  =  \partial_\alpha c(\alpha_n,0), \quad 
 v^{(n)}_0 =  \partial_\eps c(\alpha_n,0).  
\end{equation}

We also have that 
\begin{equation}
\label{equation alpha n +}
\frac{d}{d \eps} \alpha ^{+}_n(0) = 
- \frac{m \left(
DG_1 \left(\omega_{n-1}, f^{(n)}_{n-1}\right) v_{n-1}^{(n)}
\right)
}{\rule{0ex}{3.5ex}  
D \widehat{G}_1 \left(f^{(n)}_{n-1}\right) u_{n-1}^{(n)}} ,
\end{equation}
and 
\begin{equation}
\label{equation alpha n -}
\frac{d}{d \eps} \alpha ^{-}_n(0) = 
- \frac{M \left(
D G_1 \left(\omega_{n-1}, f^{(n)}_{n-1}\right) v_{n-1}^{(n)}
\right)
}{\rule{0ex}{3.5ex}  
D \widehat{G}_1 \left(f^{(n)}_{n-1}\right) u_{n-1}^{(n)}} ,
\end{equation}
where $m$, $M$, $G_1$ and $\widehat{G}_1$ are given by equations 
(\ref{equation definition of minim}), (\ref{equation definition of maximum}), 
(\ref{equation definition G1}) and (\ref{equation definition widehatG1}). 
\end{thm}

\begin{rem}
\label{remark avoiding second part of conjecture H2)}
If the second part of conjecture {\bf \ref{conjecture H2}} is omitted, then we
can adapt the result to be asymptotically valid. Recall that we 
have an open neighborhood of the fixed point $\Phi$ where the 
renormalization operator is differentiable. The uncoupled 
family
$\RR^n\left(\{c(\alpha,\eps)\}_{(\alpha,\eps)\in A} \right)$ is 
contracted towards $W^u(\Phi,\RR)$, then replacing the family of maps
$\{c(\alpha,\eps)\}_{(\alpha,\eps)\in A}$  
by $\RR^n\left(\{c(\alpha,0)\}_{(\alpha,\eps)\in A} \right)$ it 
would be close enough to $\Phi$ to be differentiable.  
On the other hand the manifolds $\Sigma_n$ accumulate to $\Phi$ when $n$ 
goes to $\infty$, then the functions $G_1$ and $\widehat{G}_1$ associated 
to the manifold $\Sigma_1$ must be replaced by suitable 
function $G_n$ and $\widehat{G}_n$ associated to the 
manifolds $\Sigma_n$.  These would give place to a more restrictive 
result, but it would be valid for the 
asymptotic estimates that are done in \cite{JRT11b}. 
\end{rem}

Now we can go back to hypothesis {\bf H2}, which is not intuitive, 
but we can introduce a stronger condition which is much more easy to 
check. Moreover this condition is automatically satisfied by maps like 
the Forced Logistic Map. 

Consider a family of maps 
$\{c(\alpha,\eps)\}_{(\alpha,\eps)\in A}$ as before, 
satisfying hypothesis  {\bf H1}. Consider the following hypothesis on the map.

\begin{description}
\item[H2')]  The family $\{c(\alpha,\eps)\}_{(\alpha,\eps)\in A}$ is such that
the quasi-periodic perturbation $\partial_\eps c(\alpha,0)$ belongs to the set
\begin{equation}
\label{equation space B1} 
\BB_1 : = \big\{ f \in \BB |\text{ } f(\theta, x) = u(x) \cos(2\pi \theta) + 
v(x) \sin(2\pi\theta),  
\text{ for some } u,v\in \RHH(\W)\big\},
\end{equation} 
for any value of $\alpha$ (with $(\alpha,0)\in A$). 
Here $\RHH(\W)$ denotes the real holomorphic maps on the set $\W$,
and  $\W$ is the set given by hypothesis  {\bf H0}
(see section \ref{section definition and basic properties} for more details).
\end{description}

Then we have the following result. 

\begin{prop}
\label{propostion H2' implies H2} 
If a family $\{c(\alpha,\eps)\}_{(\alpha,\eps) \in A}$
satisfies {\bf H1} and {\bf H2'} then it satisfies
{\bf H2}.
\end{prop}

On the other hand, we have the following propositions which allow 
us to compute explicitly the derivative of $G_1$ and $\hat{G}_1$.

\begin{prop}
\label{proposition derivative x(f)} 
Assume that we have $f_1 \in \Sigma_1$. Consider $V$ a neighborhood of 
$f_1$ (in the topology of $\BB$) and $x:V\rightarrow \RHH(\B_\rho,\W)$ 
the map such that $x(f)$ is a two periodic invariant curve of $f$ (with 
rotation number $\omega$). Then we have 
\[ \left[ D_f x(f_1) h  \right] (\theta) =  
D_x f_1(1) h(\theta -2 \omega, 0) + h(\theta-\omega,1).\]
\end{prop}

\begin{prop}
\label{proposition derivative G1} 
Assume that we have $f_1 \in \Sigma_1$. Consider $V$ a neighborhood of 
$f_1$ (in the topology of $\BB$). Consider also the map $G_1$ defined 
in (\ref{equation definition G1})  
where $x(g)$ is the two periodic invariant curve of the map. 
Then we have 
\begin{equation}
\label{equation defivative G1}
\left[ D_g G_1 (f_1) h\right](\theta)  = 
D_x f_1(1) \Big[ D^2_{x^2} f_1(0) \Big( D_x f_1(1) h(\theta-2\omega, 0) + 
h(\theta -\omega,1) \Big)  +  D_x h (\theta,0) \Big].
\end{equation}
\end{prop}

Using the last propositions one can be more explicit on the directions 
of the reducibility-loss bifurcations given by formulas 
(\ref{equation alpha n +}) and (\ref{equation alpha n -}) 

\begin{cor}
\label{corollary non-reducibility directions} 
Assume that the same hypotheses of 
theorem \ref{thm existence of non-reducibility direction} are satisfied and 
consider additionally  the sequences (\ref{equation sequences 
corollary directions}) for the same initial terms 
(\ref{equation sequences directions initial}) of the theorem. 
Then we have
\begin{equation}
\label{equation alpha n + corollary}
\frac{d}{d \eps} \alpha ^{+}_n(0) = 
- \frac{ \displaystyle
 \min_{\theta \in \T} 
\left[ \begin{array}{r}
D^2_{x^2}  f_{n-1}^{(n)}(0) \Big( D_x  f_{n-1}^{(n)}(1) 
v_{n-1}^{(n)}(\theta-2\omega_{n-1}, 0) + v_{n-1}^{(n)} (\theta -\omega_{n-1},1) \Big) \\  +  
D_x v_{n-1}^{(n)} (\theta,0)  \end{array}
\right]}
{
D_{x^2}  f_{n-1}^{(n)}(0) \Big( D_x  f_{n-1}^{(n)}(1) 
u_{n-1}^{(n)}(0) + u_{n-1}^{(n)} (1) \Big)  +  
D_x u_{n-1}^{(n)} (0) 
}.
\end{equation}
and 
\begin{equation}
\label{equation alpha n - corollary}
\frac{d}{d \eps} \alpha ^{-}_n(0) = 
- \frac{ \displaystyle
 \max_{\theta \in \T} 
\left[ \begin{array}{r} 
D^2_{x^2}  f_{n-1}^{(n)}(0) \Big( D_x  f_{n-1}^{(n)}(1) 
v_{n-1}^{(n)}(\theta-2\omega_{n-1}, 0) + v_{n-1}^{(n)} (\theta -\omega_{n-1},1) \Big) \\ 
 +  D_x v_{n-1}^{(n)} (\theta,0) \end{array}
\right]}
{
D_{x^2}  f_{n-1}^{(n)}(0) \Big( D_x  f_{n-1}^{(n)}(1) 
u_{n-1}^{(n)}(0) + u_{n-1}^{(n)} (1) \Big)  +  
D_x u_{n-1}^{(n)} (0) 
},
\end{equation}
\end{cor}

These explicit formulas will be used in \cite{JRT11c} to compute numerically 
the directions $\frac{d}{d \eps} \alpha ^{+}_n(0)$ 
with the use of a discretization of the operators $\RR$ and $\TT_\omega$. 


{\bf Proofs}

\begin{proof}[Proof of theorem \ref{thm existence of non-reducibility direction}] 
We consider only the case involving $\frac{d}{d \eps} \alpha ^{+}_n(0)$ since 
the other is completely analogous. When the hypothesis {\bf H1} is satisfied 
there exists a sequence of parameter values $\{\alpha_n\}_{n\in Z_+}$, such that 
the critical point of the map $c(\alpha_n,0)$ is a $2^n$-periodic orbit. 
In other words, we have that $c(\alpha_n,0)\in \Sigma_n$. Moreover 
the map $c(\alpha_n,0)$ is (at least) $n-1$ times renormalizable in the 
one dimensional sense. Now, due to the perturbative construction of 
the q.p. renormalization operator $\TT_\omega$ we have that there 
exists a neighborhood $U_n$ (on $\BB$) of $c(\alpha_n,0)$ such that 
any map in $U_n$ is $n-1$ times renormalizable (in the q.p. sense).

Using theorem \ref{theorem local chard of upsilon n}
we have that $\Upsilon_n^+(\omega)$ is locally given as 
\[
\Upsilon^+_n(\omega) \cap U_n = \{g \in U_n | \thinspace G^+_1(T^{n-1}(\omega, g)) =0 \},
\]
where $G^+_1=m\circ G_1$, with $m$ the minimum function 
 (\ref{equation definition of minim})
and $G_1$
is given by (\ref{equation definition G1}). 
 
Consider now a neighborhood $A_n=(a_n,b_n)\times[0,c_n)$ of the parameter value 
$(\alpha_n,0)$, which is small enough to have  
$\{c(\alpha,\eps)\}_{(\alpha,\eps)\in A_n} \subset U_n$. We can define 
the following map  
\[
\begin{array}{rccc}
g_n:& (a_n,b_n)\times[0,c_n)  &\rightarrow & \R \\
\rule{0pt}{3ex} &  (\alpha,\eps) & \mapsto & m\circ G_1(T^{n-1}(\omega,c(\alpha,\eps))). 
\end{array}
\]

Then we have that 
\[\{c(\alpha,\eps)\}_{(\alpha,\eps)\in A_n} \cap 
\Upsilon^+_n = \{ 
(\alpha,\eps)\in  A_n | \thinspace g_n(\alpha,\eps) =0 \}. 
\]

The proof of the theorem follows applying the implicit function 
theorem to the function $g_n$ at the point $(\alpha_n,0)$. 
With this aim, let us describe $T^{n-1}(\omega,c(\alpha,\eps))$ with some more detail.

Recall that the family $\{c(\alpha,\eps)\}_{(\alpha,\eps)\in A}$
 uncouples, therefore for any $(\alpha,\eps)\in A$ 
we can write $c(\alpha,\eps) = c_0(\alpha) + \eps c_1(\alpha,\eps)$. 

From now on assume that $(\alpha,\eps)$ are such that 
$f_k = \RR^k\left(c(\alpha,\eps)\right)$ is contained 
in the set $U$ of conjecture {\bf \ref{conjecture H2}}. Using the differentiability 
of $T$ (with respect its component on $\BB$) we have that 
\[
T^{n-1}(\omega_0, c(\alpha,\eps)) = 
(\omega_{n-1} , \RR^{n-1}(c_0(\alpha) ) + \eps H_0(\omega_0,\alpha,\eps)),
\]
with 
\[
H_0(\omega_0, \alpha,0) = 
 D\TT_{\omega_{n-2}}(f_{n-2}) \cdots D\TT_{\omega_0}(f_0) \partial_\eps c(\alpha_n,0),
\]
where $\omega_k= 2^k\omega_0$ and $f_k = \RR^k(c(\alpha,0))$, for $k=0,\dots,n-1$.  

Applying now the differentiability of $G_1$ we have 
\[
G_1(T^{n-1}(\omega_0, c(\alpha,\eps)))= 
G_1(\omega_{n-1}, \RR^{n-1}(c_0(\alpha))) + 
\eps H_1(\omega_0, \alpha,\eps), 
\]
with 
\[
H_1(\omega_0,\alpha,0) = DG_1 (\omega_{n-1}, f_{n-1}) H_0(\omega,\alpha,0),
\]
where $\omega_{n-1}=2^{n-1} \omega_0$ and $f_{n-1} = \RR^{n-1}(c(\alpha,0))$. 

Note that $\RR^n(c_0(\alpha))$ is an uncoupled map, therefore 
$G_1\left(\omega_{n-1}, \RR^{n-1}(c_0(\alpha))\right) $ as a function of 
$C^\omega(\T,I)$ is constant, and then its minimum is equal to this constant. 
Actually we have that 
$\displaystyle m\left( G_1(\omega_{n-1}, \RR^{n-1}(c_0(\alpha)))\right) = 
\widehat{G}_1(\RR^{n-1}(c_0(\alpha))$. On the other hand, note 
that $\eps\geq 0 $ for any $(\alpha,\eps)\in A_n$.
Using these two facts to $m$, the minimum operator 
(\ref{equation definition of minim}), we have 
\begin{eqnarray}
g_n(\omega,\alpha,\eps)  & = & 
m\left( G_1(\omega_{n-1}, \RR^{n-1}(c_0(\alpha))) + 
\eps H_1(\omega_0, \alpha,\eps) \right)   \nonumber \\ 
& = & \rule{0ex}{2ex}
\label{equation intermediate expresion of gn} 
\widehat{G}_1(\RR^{n-1}(c_0(\alpha)))+m\left(\eps H_1(\omega_0, \alpha,\eps)\right) 
\\
& = & \rule{0ex}{2ex}
\label{equation final expresion of gn} 
\widehat{G}_1(\RR^{n-1}(c_0(\alpha)))+\eps \left( m\circ H_1(\omega_0,\alpha,\eps)\right). 
\end{eqnarray}
Where $\widehat{G}_{1}$ is defined like in 
(\ref{equation definition widehatG1}) in a suitable 
neighborhood of $\RR^{n}(c(\alpha))$. 

Our aim is to apply the IFT to $g_n$ at the point $(\alpha,\eps)
=(\alpha_n,\eps)$. Note that $c(\alpha_n,0)$ accumulates to $W^s(\Phi,\RR)$ 
when $n$ grows. Then the sequence $\{f_k^{(n)}\}_{0\leq k<n} $ attains to 
$W^s(\Phi,\RR)\cup W^u(\Phi,\RR)$ when $n$ grows. Therefore 
there exists $n_0$ such that, for any $n\geq n_0$, $f_k^{(n)}\in U$ 
for $k=0,\dots,n-1$, where $U$ is the set given by conjecture {\bf \ref{conjecture H2}}. 
Then the function $g_n(\omega,\alpha,\eps)$ is differentiable in a 
neighborhood of $(\alpha_n,0)$ if  $H_1(\omega_0, \alpha_n,0))$ has a 
unique non-degenerate minimum. Note that
\begin{equation}
\label{equation function H1} 
H_1(\omega_0, \alpha_n,0) = DG_1 (\omega_{n-1}, f_{n-1})  
D\TT_{\omega_{n-2}}(f_{n-2}) \cdots D\TT_{\omega_0}(f_0) \partial_\eps c(\alpha_n,0),
\end{equation}
which actually corresponds to the hypothesis {\bf H2}, which is satisfied.

Recall that if $c(\alpha_n,0)\in \Sigma_n$ then 
we have $g_n(\omega,\alpha_n,0)=0$. Therefore, 
to apply the Implicit Function Theorem (IFT for short) to $g_n$,  we need to 
check that $\partial_\alpha g_n(\omega,\alpha_n,0) \neq 0$. 
Note also that $\widehat{G}_1$ is the function which gives locally the 
manifold $\Sigma_1$, therefore 
the condition  $\partial_\alpha g_n(\omega,\alpha_n,0) \neq 0$ is 
equivalent to require that ${\RR^{n-1}(c_0(\alpha))}_{\alpha\in (a_n,b_n)}$ 
intersects transversally the manifold $\Sigma_1$ at
the point $\RR^{n-1}(c_0(\alpha_n))$. 
This condition is satisfied due the hypothesis {\bf H1} which 
requires the family 
$\{c(\alpha)\}_{\alpha\in[a,b]}$  to cross transversely each manifold 
$\Sigma_n$ for any $n$. 

Then the IFT ensures us that there exists a neighborhood 
$\tilde{A}= (\tilde{a}_n,\tilde{b}_n)\times[0,\tilde{c_n})$ and a function $\alpha_n(\omega) : [0,d_n)\rightarrow R)$ 
such that 
\[
\{c(\alpha,\eps)\}_{(\alpha,\eps) \in \tilde{A}} \cap \Upsilon_n(\omega) 
= \{c(\alpha_n(\omega,\eps),\eps)\}_{\eps\in[0,d_n)}. 
\]
Moreover we have that 
\[
\frac{d}{d \eps} \alpha ^{+}_n(0) = - 
\frac{\partial_{\eps} g_n(\omega,\alpha_n,0) }
{\partial_{\alpha} g_n(\omega,\alpha_n,0) }. 
\]

Finally we can use equation (\ref{equation final expresion of gn}) and 
(\ref{equation function H1}) to compute $\partial_{\eps} 
g_n(\omega,\alpha_n,0)$ and $\partial_{\alpha} g_n(\omega,\alpha_n,0)$. Then,
it follows
\[
\frac{d}{d \eps} \alpha ^{+}_n(0) = 
- \frac{ \displaystyle
 m\left(
D G_1 \left(\omega_{n-1}, f^{(n)}_{n-1}\right) 
D\TT_{\omega_{n-2}}\left(f^{(n)}_{n-2}\right) \cdots 
D\TT_{\omega_0}\left(f^{(n)}_0 \right) \partial_\eps c(\alpha_n,0) 
\right)}
{\rule{0ex}{3.5ex}  
D \widehat{G}_1 \left(f^{(n)}_{n-1}\right) 
D\RR \left(f^{(n)}_{n-2}\right) \cdots D\RR \left(f^{(n)}_0\right) 
\partial_\alpha c(\alpha_n,0)}, 
\]
which indeed is  equivalent to (\ref{equation alpha n +}).
\end{proof}


\begin{rem}
In theorem \ref{thm existence of non-reducibility direction} 
the parameter space has been set up in such a way 
that $\eps$ is greater or equal zero for any values of 
the parameter. In the proof above the IFT 
is applied to 
a neighborhood of  $(\alpha,\eps) = (\alpha_n,0)$. But the theorem is 
not completely applicable in its usual form (see for example \cite{Die69}, 
because we do not have the 
derivative defined w.r.t. directions with $\eps$ negative. To bypass 
this difficulty, we can consider the map $g_n(\omega,\alpha,\eps)$ written as 
in (\ref{equation final expresion of gn}). Then we can 
extend the map symmetrically as
\[g_n(\omega,\alpha,\eps) =  \widehat{G}_1(\RR^{n-1}(c_0(\alpha)))+\eps \left( m\circ H_1(\omega_0,\alpha,- \eps)\right). 
\]
for any $\eps<0$. This extension of the map is enough to have 
a $C^1$ map and to apply the IFT. If more differentiability 
is required, then other extensions should be considered.
\end{rem}

\begin{figure}[t]
\centering
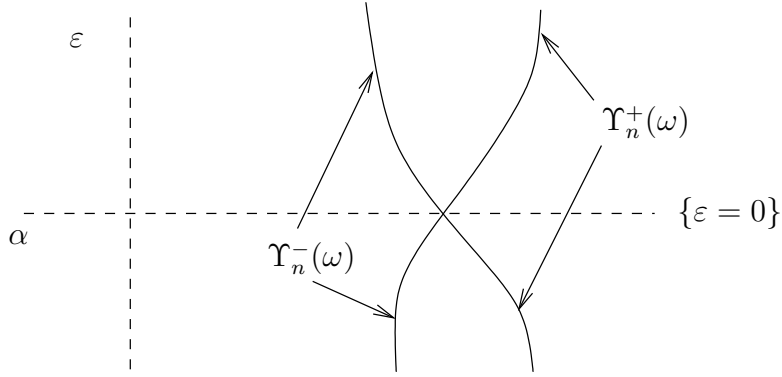
\caption{Representation of the intersection of the sets
$\Upsilon_n^+ (\omega)$ and $\Upsilon_n^- (\omega)$ 
with a two parametric family of maps. The plane 
represented correspond to the plane of parameters $(\alpha,\eps)$ 
of the family.}
\label{Esquema Varietats Upsilon}
\end{figure}

\begin{rem} Recall that the domain of the parameters has been considered of the 
form $A= [a,b]\times[0,c]$. One might replace this domain for one of 
the form $\tilde{A}= [a,b]\times[-c,c]$. Then one can redo the 
proof of theorem \ref{thm existence of non-reducibility direction} with the 
new set up. But then, in the step from (\ref{equation intermediate expresion of gn}) 
to (\ref{equation final expresion of gn}) one can not apply the IFT, because
the function $g_n(\omega, \alpha,\eps)$ is not differentiable. This is because the 
minimum is replaced by a maximum for $\eps<0$ which makes the function only $C^0$. 
What one can do is to split $\tilde{A}= [a,b]\times [-c,c]$ into 
$\tilde{A}_-=[a,b]\times[-c,0]$ and $\tilde{A}_+=[a,b]\times[0,c]$. Then 
one can apply theorem \ref{thm existence of non-reducibility direction} 
twice, and then we obtain four different bifurcation curves emerging form 
the same point $(\alpha_n,0)$. What actually happens is that 
we have two smooth curves, crossing at $(\alpha_n,0)$, but the curves defined
by $\Upsilon^+_n(\omega) \cap \{c(\alpha,\eps)\}_{(\alpha,\eps) \in \tilde{A}}$ and 
$\Upsilon^-_n(\omega) \cap \{c(\alpha,\eps)\}_{(\alpha,\eps) \in \tilde{A}}$
swap their position when one crosses $\eps=0$. This is illustrated 
in figure \ref{Esquema Varietats Upsilon}.

\end{rem}

\begin{proof}[Proof of proposition \ref{propostion H2' implies H2}]
Applying proposition \ref{proposition invariance of Fourier node spaces}
on section \ref{section definition and basic properties}, we have
that the space $\BB_1$ is invariant by $D\TT_\omega(f)$ for any $\omega \in \T$ and
$f \in \BB_0$ in a neighborhood of the fixed point. Consequently, we have that
$v_{n-1}= D\TT_{\omega_{n-2}}\left(f^{(n)}_{n-2}\right) \cdots 
D\TT_{\omega_0}\left(f^{(n)}_0 \right) \partial_\eps c(\alpha_n,0)$
belongs to $\BB_1$, where $\omega_k=2^k \omega_0$ and
$f_k=\RR^k(c(\alpha,0))$. Finally, note that when we consider
$DG_1(\omega_{n-1}, f_{n-1}) v_{n-1}$ we
obtain a periodic function of the form $A \cos(2\pi\theta) + 
B \sin(2\pi\theta)$, which will have a unique non-degenerate
minimum (and maximum) if $A$ and $B$ are not simultaneously zero.
\end{proof}

\begin{proof}[Proof of proposition \ref{proposition derivative x(f)}] 
We have that the map $x:V\rightarrow \RHH(\B_\rho, \W)$ is obtained 
applying the IFT to the map
\[
\begin{array}{rccc}
F_1:& V \times \RHH(\B_\rho,\C) &\rightarrow & \RHH(\B_\rho,\C)  \\
\rule{0pt}{3ex} & (g,x )  & \mapsto & [F(g,x)](\theta):= 
g(\theta + \omega, g(\theta, x(\theta))) - x(\theta + 2 \omega). 
\end{array}
\]
at the point $(f_1, x(f_1))$. From the IFT we also know that 
\begin{equation}
\label{equation derivative x via IFT} 
D_h x[f] h = - (D_x F_1(f,x[f]))^{-1} \circ (D_f F_1(f,x[f])) h.
\end{equation}

Differentiating $F_1$ we have 
\[
\begin{array}{rccc}
D_x F_1(g,x[g]):& \RHH(\B_\rho,\C) & \rightarrow & \RHH(\B_\rho,\C) \\
\rule{0pt}{3ex} & l   & \mapsto & 
D_x g\big(\theta +\omega, g(\theta, x[g](\theta))\big)  
D_x g\big(\theta, x[g](\theta)\big) l(\theta) - l(\theta + 2 \omega). 
\end{array}
\]
and 
\[
\begin{array}{rccc}
D_f F_1(g,x[g]):& \RHH(\B_\rho\times \W,\C)  &\rightarrow & \RHH(\B_\rho,\C)  \\
\rule{0pt}{4.5ex} & h  & \mapsto & 
\begin{array}{c} 
D_x g\big(\theta + \omega, g(\theta, x[g](\theta))\big) h (\theta, x[g](\theta)) + \\
\rule{0pt}{2.5ex}h\big(\theta +\omega, g(\theta, x[g](\theta))\big). \end{array}
\end{array}
\]

From the fact that $f_1\in \BB_0$ it follows that its critical point is
$x=0$ and $f_1(0)=1$. From $f\in \Sigma_1$ we have that the critical point is 
a two periodic orbit, therefore $x(f_1) = 0$ and $f(f(0))=f(1)=0$.  Using these 
properties, given $h\in \RHH(\B_\rho\times \W,\C)$ we have 
\[
\Big[D_f F(f_1, x[f_1]) h \Big](\theta) = D_x f_1(1) h(\theta,0) + h(\theta+\omega,1).
\]
On the other hand, given $\ell \in  \RHH(\B_\rho\times \W,\C) $ we have that 
\[
\Big[D_x F(f_1, x[f_1]) \ell \Big](\theta) = - \ell(\theta +2\omega). 
\]
Note that this operator is easily invertible, and its inverse is given by 
\[ \Big[ (D_x F_1(f,x[f]))^{-1} \ell  \Big](\theta) = -\ell(\theta -2\omega). \]

Now we can use (\ref{equation derivative x via IFT}) and the last two equations 
to conclude that 
\[
\Big[D_h x(f_1) h\Big](\theta) = D_x f_1(1) h(\theta -2 \omega,0) + h(\theta-\omega,1). 
\]
\end{proof}

\begin{proof}[Proof of proposition \ref{proposition derivative G1}]
Using the chain rule on map (\ref{equation definition G1}) is it 
not hard to see that 
\[
\begin{array}{rcl} 
\Big[ D_g G_1(g) h \Big] (\theta) & = & 
D_x g(\theta + \omega, g(\theta,x(\theta))) \Big( 
D^2_{x^2} g(\theta, x[g](\theta))\big[ D_h x(g) h\big](\theta)  
+ D_x h (\theta, x[g](\theta)) 
\Big) \\
 & & + D_x g(\theta, x[g](\theta) ) H(g,x[g],h) (\theta), 
\end{array}
\]
where $H(g,x[g],h)$ is an expression on $g$, $x[g]$ and $h$ that 
can be explicitly computed but here it has no interest since 
it will be cancelled out further on. 

From the fact that $f_1\in \BB_0$, it follows that its critical point is at
$x=0$ (i.e. $D_x f_1 (0) =0$) and  $f_1(0)=1$. 
From $f_1\in \Sigma_1$ we have that the critical point is
a two periodic orbit, therefore $x(f_1) = 0$ and $f_1(f_1(0))=f_1(1)=0$. Using this
 properties and the equation above we have 
\[
\Big[ D_g G_1(f_1) h \Big] (\theta) = 
D_x f_1 (1) \Big( 
D^2_{x^2} f_1(0) \big[ D_f x(f_1) h\big](\theta)   + D_x h(\theta, 0)
\Big). 
\] 
Finally, we can use proposition \ref{proposition derivative x(f)} to 
compute $\big[ D_f x(f_1) h\big](\theta)$, then the stated result holds.
\end{proof}

\begin{proof}[Proof of corollary \ref{corollary non-reducibility directions}] 
From theorem 
\ref{thm existence of non-reducibility direction} it follows that 
\begin{equation}
\label{equation alpha n + proof of corollary}
\frac{d}{d \eps} \alpha ^{+}_n(0) =
- \frac{ \displaystyle  m\left( 
D G_1 \left(\omega_{n-1}, f^{(n)}_{n-1}\right) v_{n-1}^{(n)} \right)}
{  D \widehat{G}_1 \left(f^{(n)}_{n-1}\right) u_{n-1}^{(n)}}. 
\end{equation} 

We can apply now proposition \ref{proposition derivative G1} to compute 
the derivative of $G_1$. Note that  $\widehat{G}_1$ can be seen as 
$G_1$ restricted to $\BB_0$, then the proposition \ref{proposition derivative G1} 
is also applicable to $\widehat{G}_1$. Then we have 
\begin{align*}
D G_1 \left(\omega_{n-1}, f^{(n)}_{n-1}\right) v_{n-1}^{(n)} = & 
 D_x  f_{n-1}^{(n)}(1)
\Big[ D_{x^2}  f_{n-1}^{(n)}(0) 
\Big( D_x  f_{n-1}^{(n)}(1)  v_{n-1}^{(n)}(\theta-2\omega_{n-1}, 0)  \\ 
& 
+ v_{n-1}^{(n)} (\theta -\omega_{n-1},1) 
\Big)   +  
D_x v_{n-1}^{(n)} (\theta,0), 
\Big]
\end{align*} 
and
\[
D \widehat{G}_1 \left(f^{(n)}_{n-1}\right) u_{n-1}^{(n)} = 
 D_x  f_{n-1}^{(n)}(1) 
\Big[D_{x^2}  f_{n-1}^{(n)}(0) \Big( D_x  f_{n-1}^{(n)}(1) 
u_{n-1}^{(n)}(0) + u_{n-1}^{(n)} (1) \Big)  +  
D_x u_{n-1}^{(n)} (0) \Big].
\]

Finally let us remark that $f_{n-1}^{(n)}$ belongs to $\BB_0$, therefore we have 
that $D_x f_{n-1}^{(n)}(x)x <0$ for any $x\in I\setminus\{0\}$. Concretely we have 
$D_x f_{n-1}^{(n)}(1)<0$. If we replace the values of 
 $ D G_1 \left(\omega_{n-1}, f^{(n)}_{n-1}\right) v_{n-1}^{(n)}$ 
and $ D \widehat{G}_1 \left(f^{(n)}_{n-1}\right) u_{n-1}^{(n)} $ in 
(\ref{equation alpha n + proof of corollary}), then when we
simplify the value $D_x f_{n-1}^{(n)}(1)<0$ the minimum becomes a maximum. 
\end{proof}


\begin{appendix}

\section{The minimum function}
\label{appendix minimum function}

In this appendix we give some basic properties of the minimum of a function 
$f:\T \rightarrow \R$ as an operator, with $\T= \R/\Z $ the one dimensional real torus. 
To work in the same topology that in the rest of the paper 
we will consider $f\in \RHH(\B_\rho,\C)$ the space of real analytic  
functions from $\B_\rho$ to $\C$, and continuous on the 
closure of $\B_\rho$, with $\B_\rho$  a band of width $\rho$ around
the real torus $\T$. 
In other words we want to study the operator
\[
\begin{array}{rccc}
m:&  \RHH(\B_\rho,\C) &\rightarrow & \R \\
\rule{0pt}{3ex} & g & \mapsto & \displaystyle \min_{\theta \in \T} g(\theta). 
\end{array}
\]

More concretely we focus on the differentiability of the 
map $m$. Note that in the space of holomorphic functions, it has 
no sense to consider the minimum of a function. Nevertheless, 
for maps in $\RHH(\B_\rho,\C)$ 
we have that the image of real numbers are real numbers, 
then we can consider the minimum in the real torus. 

Concerning to differentiability of the minimum as an operator, one has 
the following result.

\begin{prop}
\label{proposition diferentiability of minimum}
Let $g_0\in \RHH(\B_\rho,\C)$ be a function such that its global minimum 
in $\T$ is attained only at one value $\theta_0\in \T$ and it is not degenerate, i.e. 
$g''(\theta_0)>0$. Then there is a function $G:  \RHH(\B_\rho,\C) \rightarrow \C$  such 
that $G(g)=m(g)$ for any $g\in  \RHH(\B_\rho,\C)$ and $G$ is infinitely many times 
differentiable in a small neighborhood of $g_0$.
In other words we have that $m$ can be extended to a differentiable 
function $G:\RHH(\B_\rho,\C) \rightarrow \C$ around $g_0$. 

Moreover, the derivative of $G$ in $g_0$ is given by
\[
DG(g_0) \thinspace g_1 = g_1(\theta_0). 
\]
\end{prop} 

\begin{proof}
To prove the proposition we construct the map $G$ using 
the IFT. 

Consider the auxiliary function 
\[
\begin{array}{rccc}
F:& \RHH(\B_\rho,\C)\times \T &\rightarrow & \C \\
\rule{0pt}{3ex} & (g,\theta)  & \mapsto & g'(\theta).
\end{array}
\]
We have that the function $F$ corresponds to the derivative of $g$ composed 
with the evaluation in the point $\theta$. Both the derivative and the evaluation map 
are $C^\infty$ functions (they are bounded 
linear operators, therefore they are infinitely 
differentiable), 
then  its composition, which is $F$,  is also $C^\infty$. 

Consider $g_0$ as in the hypothesis of the proposition. Then we have that 
there exist a $\theta_0\in \T$ such that $(g_0,\theta_0)$ is a zero of 
$F$. By hypothesis we have that $\theta_0$  is not a degenerate minimum, then 
\[
D_\theta F(g_0,\theta_0) = g_0''(\theta_0) >0,
\]
We can apply now the IFT, consequently we have that there exist an open neighborhood 
$U\subset  \RHH(\B_\rho,\C)$ of $g_0$  and a $C^\infty$ function 
\begin{equation}
\label{function theta g}
\begin{array}{rccc}
\theta:& U &\rightarrow & \T \\
\rule{0pt}{3ex} & g & \mapsto & \theta(g),
\end{array}
\end{equation}
such that $F(g,\theta(g))=0$ for any $g\in U$. Concretely  we have 
that $\theta(g)$ is a local minimum for any $g\in U$. Recall that by 
hypothesis we have also that $\theta_0$ is a global minimum of $g_0$ and 
that it is unique. Then, reducing $U$ to a smaller neighborhood if necessary, we 
have that $\theta(g)$ is also the unique global minimum of $g$, for 
any $g\in U$. 

Let us consider the evaluation map $\operatorname{Ev}:  \RHH(\B_{\rho'},\C) \times 
\T \rightarrow \C$ the evaluation map (in a point of the real torus 
$\T\subset \B_{\rho'}$) with $\rho'$ a value 
$0<\rho'<\rho$. Let $\theta$ be the function 
(\ref{function theta g}) defined by  the IFT. For any $g\in U$ we 
can define the map $G$ in the statement of the proposition 
\ref{proposition diferentiability of minimum}
as $G(g) = \operatorname{Ev}(g',\theta(g))$. With this definition we have that $G(g)= m(g)$ for 
any $g\in U$. Therefore the minimum 
is a  $C^\infty$ function in a neighborhood of $g_0$. Note that 
in the definition of the evaluation map $\operatorname{Ev}$ we have considered
$\RHH(\B_{\rho'},\C)$ as its domain with $0<\rho'<\rho$. This is needed in 
order to ensure that $g'$ is a bounded function, then one has $G$ is 
a bounded operator. 

To finish we compute the derivative of the minimum function. Is not difficult to 
see that $D_\theta Ev(f,\theta) = f'(\theta)$ and $ D_g Ev(f,\theta) g_1 = g_1(\theta)$. 
Using the chain rule we have 
\[
DG(g) g_1 = D_g Ev(g,\theta(g)) g_1 + D_\theta Ev(g,\theta(g)) D_g \theta(g) g_1.
\]
Recall that $g_0'(\theta(g_0))= g_0'(\theta_0)=0$, then we have
$D_\theta Ev(g,\theta(g)) = g'(\theta(g)) =0 $. Then it follows 
\[
DG(g_0)\thinspace g_1 = g_1(\theta_0).
\]
\end{proof}

Let us discuss the case when the hypotheses of the proposition 
\ref{proposition diferentiability of minimum} are not satisfied. 
The two main hypotheses of the proposition are the 
non-degeneracy of the minimum and the uniqueness of it. When the non-degeneracy 
condition is suppressed the proposition still being true, since the 
argument done before can be adapted changing the auxiliary function, although the 
proof becomes quite more technical. 

On the other hand, the uniqueness condition of the 
minimum is always necessary, because when 
it is not satisfied the map ceases to be Frechet differentiable.

We just want to see that it is not differentiable. For simplicity 
we will take de derivative in the reals. Let us consider the 
most degenerate case, which is when the function $g_0$ is constant. 
Then given  $g_1\in \RHH(\B_\rho,\C)$
we want to compute $Dm(g_0, g_1)$, the Gateaux derivative of $m$ 
at $g_0$ with respect to 
the $g_1$ direction. By definition we have
\[
Dm(g_0, g_1 )= \lim_{t\rightarrow 0} \frac{m(g_0+t g_1) - m(g_0)}{t}
\] whenever the limit exist. 

Indeed, when the function $g_0(\theta)$ is constant we have that 
\[\min_{\theta\in T} (g_0 + t g_1) = g_0 +  \min_{\theta\in T} t g_1 =
\left\{ 
\begin{array}{rl}
\displaystyle g_0 + t\min_{\theta\in T} g_1  & \text{ if } t\geq 0 \\
\rule{0pt}{4ex} \displaystyle 
g_0 + t\max_{\theta\in T} g_1  & \text{ if } t\leq 0 
\end{array}
\right.\] 

Then the Gateaux derivative not only depends on the sign of $t$ when taking the 
limit, moreover when the sign is fixed the limit which we obtain is not even
a linear operator. It is clear that in this case the minimum 
operator is not differentiable.

\end{appendix}


\bibliographystyle{plain}

\end{document}